\documentclass[12pt,a4paper]{amsart}
\usepackage{graphicx,amssymb,amsfonts,epsfig,amsthm,a4,amsmath,url,graphicx,enumerate,verbatim}
\usepackage{datetime2}
\usepackage{fontenc}
\usepackage[usenames,dvipsnames]{color}

\newtheorem{thm}{Theorem}[section]
\newtheorem{cor}[thm]{Corollary}
\newtheorem{lem}[thm]{Lemma}

\newtheorem{prop}[thm]{Proposition}
\theoremstyle{definition}
\newtheorem{defn}[thm]{Definition}

\newtheorem{que}[thm]{Question}

\theoremstyle{remark}
\newtheorem{rem}[thm]{Remark}

\newtheorem{ex}[thm]{Example}

\numberwithin{equation}{section}

\newcommand{\Z}{\mathbf{Z}}

\newcommand{\N}{\mathbf{N}}

\newcommand{\R}{\mathbf{R}}
\newcommand{\K}{\mathbf{K}}
\newcommand{\C}{\mathbf{C}}
\newcommand{\KP}{\mathbf{KP}}
\newcommand{\CP}{\mathbf{CP}}
\newcommand{\Q}{\mathbf{Q}}
\newcommand{\dom}{\operatorname{dom}}
\newcommand{\ran}{\operatorname{ran}}
\newcommand{\vect}{\operatorname{span}}
\newcommand\polar[2]{\ensuremath{\|#1_{#2}\|}}
\newcommand\dual[2]{\ensuremath{\langle #1,#2\rangle}}
\newcommand\prep[1]{\ensuremath{{}^\circ #1}}
\newcommand\Banach{\mathcal{X}}
\newcommand\Op{\mathcal{T}}
\newcommand\Opf{\mathcal{T}_f}
\newcommand\Opff{\mathcal{T}_{f,<\infty}}
\newcommand\Opfn{\mathcal{T}_{f,n}}

\newcommand{\cD}{\mathcal{D}}
\newcommand{\cF}{\mathcal{F}}
\newcommand{\GL}{\textrm{GL}}
\newcommand{\U}{\textrm{U}}

\begin{document}
\title{A duality operators/Banach spaces}
\date{\DTMnow}
\author{Mikael de la Salle}
\begin{abstract}
  Given a set $B$ of operators between subspaces of $L_p$ spaces, we characterize the operators between subspaces of $L_p$ spaces that remain bounded on the $X$-valued $L_p$ space for every Banach space on which elements of the original class $B$ are bounded.

  This is a form of the bipolar theorem for a duality between the class of Banach spaces and the class of operators between subspaces of $L_p$ spaces, essentially introduced by Pisier. The methods we introduce allow us to recover also the other direction --characterizing the bipolar of a set of Banach spaces--, which had been obtained by Hernandez in 1983.
\end{abstract}
\maketitle 

\section{Introduction}

All the Banach spaces appearing in this paper will be assumed to be separable, and will be over the field $\K$ of real or complex numbers.


The \emph{local theory of Banach spaces} studies infinite dimensional Banach spaces through their finite-dimensional subspaces. For example it cannot distinguish between the (non linearly isomorphic if $p \neq 2$ \cite[Theorem XII.3.8]{Banach32}) spaces $L_p([0,1])$ and $\ell_p(\N)$, as they can both be written as the closure of an increasing sequence of subspaces isometric to $\ell_p(\{1,\dots,2^n\})$~: the subspace of $L_p([0,1])$ made functions that are constant on the intervals $(\frac{k}{2^n},\frac{k+1}{2^n}]$, and the subspace of $\ell_p(\N)$ of sequences that vanish oustide of $\{0,\dots,2^n-1\}$ respectively. 

The relevant notions in the local theory of Banach spaces are the properties of a Banach space that depend only on the collection of his finite dimensional subspaces and not on the way they are organized. Said differently, the properties that are inherited by finite representability. Such properties are called \emph{super-properties}. The central question is to understand whether one super-property implies another, see Section \ref{sec:motivation} for terminology, details and examples.

The main result is Theorem \ref{thm:bipolar_of_operator_explicit}, where a theoretical criterion is obtained for the implication of two super-properties which are moreover stable under $\ell_p$-direct sums, for some $1 \leq p < \infty$ which is fixed once and for all. A result by Hernandez \cite{MR703903,hernandezThesis} (Theorem \ref{thm:hernandez} below) can be reformulated as: a superproperty $P$ is stable under $\ell_p$-direct sums if and only if it defined by $p$-homogeneous inequalities, \emph{i.e.} if and only if there is an operator $T$ between subspaces 
$\dom(T)$ and $\ran(T)$ of $L_p$ spaces $L_p(\Omega_1,m_1)$ and $L_p(\Omega_2,m_2)$ such that $X$ satisfies $P$ if and only if for every $n$, every $f_1,\dots,f_n$ in the domain of $T$ and every $x_1,\dots,x_n \in X$, \[ \int_{\Omega_2} \| \sum_i (Tf_i)(\omega_2) x_i\|^p dm_2(\omega_2) \leq \int_{\Omega_1} \| \sum_i f_i(\omega_1) x_i\|^p dm_1(\omega_1).\] If one denotes by $\polar{T}{X}$ the (possibly infinite) norm of $T \otimes \mathrm{id}_X$ between the subspaces $\dom(T) \otimes X$ and $\ran(T) \otimes X$ of $L_p(\Omega_i,m_i;X)$, then this condition can be shortly written as $\|T_X\|\leq 1$. So our result characterizes, for two operators $S$ and $T$ between subspaces of $L_p$ spaces, when $\|T_X\|\leq 1$ implies $\|S_X\| \leq 1$.

This is a form of the bipolar theorem for a duality between the set $\Banach$ of complex separable Banach spaces up to isometry and the set $\Op$ linear operators between subspaces of $L_p$ spaces defined by the assignement $(T,X) \mapsto \polar{T}{X}$. Indeed, adapting the standard terminology for locally convex topological vector spaces (see \cite[II \S 6]{MR633754}), we define~:

\begin{defn}\label{defn:polar} If $A \subset \Banach$ is a class of Banach spaces, then its \emph{polar} $A^\circ$ is the class of operators $T \in \Op$ such that $\polar T X \leq 1$ for every $X$ in $A$.
\end{defn}

\begin{defn}\label{defn:prepolar} If $B\subset \Op$, then its polar $\prep{B}$ is the class of Banach spaces $X \in \Banach$ such that $\polar T X \leq 1$ for every $T$ in $B$.
\end{defn}

This duality is a variant of the one considered in \cite{MR2732331}, where Pisier restricts to operators between $L_p$ spaces (and not subspaces of $L_p$ spaces). If one is interested in the bipolar of a class of Banach spaces, the two dualities are very different. On the other had, a description of the bipolar for a class of operators for Pisier's duality can be obtained from our result, see Section \ref{section:comparison} for details.

In a locally convex topological space, the bipolar theorem (\cite[II \S 6]{MR633754}) states that the bipolar of a set $C$ is equal to the closed convex hull of $C \cup \{0\}$. The inclusion of the closed convex hull of $C \cup \{0\}$ in the bipolar of $C$ is obvious; the content of the theorem is the other inclusion, which follows from the Hahn-Banach theorem. The aim of this paper is to state and prove a version of the bipolar theorem in this setting, for the correct definition of ``closed convex hull''. For the bipolar of a class of Banach spaces, this is due to Hernandez. The methods we introduce allow us to give a new proof of it (see Section~\ref{section:comparison} for the duality involving operators between $L_p$ spaces).
\begin{thm}\label{thm:hernandez}(\cite{MR703903}) The bipolar $\prep{(A^\circ)}$ of a class of Banach spaces $A \subset \Banach$ is the class of Banach spaces finitely representable in the class of all finite $\ell_p$-direct sums of elements in $A$.
\end{thm}

There is also an isomorphic version of the previous result.
\begin{thm}\label{thm:hernandez_isom}(\cite{MR703903}) Let $A \subset \Banach$ and $X \in \Banach$. The following are equivalent:
  \begin{itemize}
  \item $\|T_X\| <\infty$ for every $T \in A^\circ$.
    \item $X$ is isomorphic to a space finitely representable in the class of finite $\ell_p$ direct sums of spaces in $A$, \emph{i.e.} to a space in $\prep{A}^\circ$.
  \end{itemize}

  In that case, the Banach-Mazur distance from $X$ to a space in $\prep{A}^\circ$ is equal to $\sup_{T \in A^\circ} \|T_X\|$.
\end{thm}

Our main result is the bipolar theorem for sets of operators. To state it we have to introduce some definition.

\begin{defn}\label{defn:spatial_isometry} A \emph{spatial isometry} between finite dimensional subspaces of $L_p$ spaces is a composition of isometries of the form:
\begin{itemize}
\item (Change of density) Restriction to a subspace of $L_p(\Omega,m)$ of the multiplication by a nonvanishing measurable function $h \colon \Omega \to \K^*$, \emph{i.e.} $f \in L_p(\Omega,m)  \mapsto h f \in L_p(\Omega,|h|^{-p} m)$.



\item (Equimeasurability outside of $0$) Maps of the form $T \colon \dom(T) \subset  L_p(\Omega,m) \to L_p(\Omega',m')$ such that for every finite family $f_1,\dots,f_n \in \dom(T)$ and every Borel subset $E \subset \K^n \setminus \{0\}$, $m(\{x, (f_1(x),\dots,f_n(x)) \in E\}) = m'(\{x, (T f_1(x),\dots,T f_n(x)) \in E\})$.
\end{itemize}
\end{defn}
It is not hard to prove (see Lemma \ref{lem:when_muf=mug} and Remark \ref{rem:when_muf=mug}) that every spatial isometry is of the form $C_1 E C_2$ for $C_1,C_2$ changes of phase and measure and $E$ equimeasurable outside of $0$. 

It is important that we require $0 \notin E$, as we want for example that
$f \in L_p([0,1])\mapsto f \chi_{[0,1]} \in L_p([0,2])$ is a spatial isometry.

When $p$ is not an even integer, it is known that every isometry between (separable) subspaces of $L_p$ spaces is a spatial isometry (\cite{MR611233}). The idea developped in this article allows to recover this result, and to generalize it to arbitrary $p$~: a linear map $T$ is a spatial isometry if and only if it is a regular isometry, \emph{i.e.} $\polar T X=\polar{T^{-1}}{X}=1$ for all $X$ (see Remark \ref{rem:spatial_and_regular_isometries} and Corollary~\ref{cor:spatial_and_regular_isometries}).

We can now state the version of the bipolar theorem for sets of operators. 

\begin{thm}\label{thm:bipolar_of_operator_explicit} Let $B \subset \Op$ and $T\colon \dom(T) \subset L_p(\Omega_1,m_1) \to L_p(\Omega_2,m_2)$ be a linear map, and $f_1,f_2,\dots,$ be a sequence generating a dense subspace of $\dom(T)$. The following are equivalent~:
  \begin{itemize}
  \item For every $X \in Banach$, $\sup_{S \in B} \|S_X\|\leq 1 \implies \|T_X\|\leq 1$.
  \item For every $n$ and $\varepsilon>0$, there exist
    \begin{itemize}
    \item an operator $S=S_0 \oplus S_1\oplus \dots \oplus S_k$ with $S_0$ of regular norm $1$ and $S_1\dots,S_k \in B$,
    \item spatial isometries
      \begin{align*} U \colon & \dom(U) \subset L_p(\Omega_1\times[0,1])\oplus_p L_p([0,1]) \to \dom(S),\\
      V \colon &\dom(V) = S(\ran U) \to L_p(\Omega_2\times[0,1])\oplus_p L_p([0,1]),
      \end{align*}
      \item for every $i=1\dots,n$ there are $g_i \in L_p(\Omega_1\times[0,1])$, $g'_i \in L_p(\Omega_2\times[0,1])$ and $h_i \in L_p([0,1])$
    \end{itemize}
    such that $(g_i,h_i) \in \dom(U)$, $V \circ S \circ U(g_i,h_i) = (g'_i,h_i)$ and
    \begin{align*} \left(\int_{\Omega_1 \times [0,1]} |f_i(\omega) - g_i(\omega,s)|^p dm_1(\omega) ds\right)^{\frac 1 p} & \leq \varepsilon\\
      \left(\int_{\Omega_2 \times [0,1]} |(Tf_i)(\omega) - g'_i(\omega,s)|^p dm_2(\omega) ds\right)^{\frac 1 p} & \leq \varepsilon.
    \end{align*}
    \end{itemize}
\end{thm}
It is instructing to work out explicitly a very simple case of this theorem, namely for the obvious implication $\max(\|S_X\|,\|T_X\|)\leq 1 \implies \|(T\circ S)_X\|\leq 1$. See Example~\ref{example:composition}.

In particular, we obtain the following characterization of the bipolar of a set of operators.
\begin{cor}\label{cor:bipolar_of_operator} The bipolar $(\prep{B})^\circ$ of a class $B \subset \Op$ is the smallest class $B'\subset \Op$ containing $B$ and satisfying the following properties~:
\begin{enumerate}[(i)]
\item\label{item:regular} $B'$ contains $\{T\in \Op, \sup_{X \in \Banach} \|T_X\| \leq 1\}$.
\item\label{item:ellpdirectsum} $B'$ is stable under finite $\ell_p$-direct sums.
\item\label{item:composing_by_spatial_isom} If $T \in B'$ and $U,V$ are spatial isometries then $U \circ T \circ V \in B'$.
\item\label{item:subtle} Let $T \in B'$ such that $T \colon \dom(T) \subset L_p(\Omega_1,m_1) \oplus L_p(\Omega,m) \to L_p(\Omega_1,m_1) \oplus L_p(\Omega,m)$ is of the form $(f,g) \mapsto (Sf,g)$ for some $S \in \Op$ with domain equal to the image of $\dom(T)$ by the first coordinate projection. Then $S \in B'$.
\item\label{item:limits} If $T \in \Op$ is an operator between subspaces of $L_p(\Omega,m)$ and $L_p(\Omega',m')$ and if, for every finite family $f_1,\dots,f_n$ in the domain of $T$ and every $\varepsilon>0$, there is $S \in B'$ with domain contained in $L_p(\Omega,m)$ and range contained in $L_p(\Omega',m')$ and elements $g_1,\dots,g_n \in \dom(S)$ such that $\|f_i - g_i\|\leq \varepsilon$ and $\| T f_i - S g_i \|\leq \varepsilon$, then $T \in B'$.
\end{enumerate}
\end{cor}
Note however that Theorem~\ref{thm:bipolar_of_operator_explicit} is a sense more precise than Corollary~\ref{cor:bipolar_of_operator}, as it almost says that to obtain the bipolar of $B$ from $B$, it is enough to apply the operations \eqref{item:regular}, \eqref{item:ellpdirectsum}, \eqref{item:composing_by_spatial_isom}, \eqref{item:subtle} and \eqref{item:limits} only once, and in that order. Almost because we obtain in this way all operators of the form $T \otimes \mathrm{id_{L_p([0,1])}}$ with domain $\{(\omega,s)\mapsto f(\omega)\mid f \in \dom(T)\}$ for $T \in (\prep{B})^\circ$, so one needs to apply one last time \eqref{item:composing_by_spatial_isom} to obtain all of $(\prep{B})^\circ$. This improvement is not minor. For a long time, the author was only able to prove Corollary~\ref{cor:bipolar_of_operator}, and actually expected that to construct $(\prep{B})^\circ$ out of $B$, it was necessary to iterate these operations (and in particular \eqref{item:subtle} and \eqref{item:limits}) a large number of times (even an arbitrarily large countable ordinal of times), and this ordinal number was a measurement of the difficulty of computing the bipolar of a $B$. This is closely related to the classical fact, essentially due to Banach, that, to obtain the weak-* closure of a convex subset in the dual of a separable Banach space, the number of times one needs to take limits of weak-* convergent sequences can be an arbitrary countable ordinal. That this is not the case will rely on a particularily strong form of the bipolar theorem (in the linear setting) for the weak-* topology that we prove in Proposition~\ref{prop:bipolar_with_compactness}. See the discussion in subsection \ref{subsec:polar_linear}.

As for the usual bipolar theorem, the main content of the theorem is the inclusion $\prep B^\circ  \subset B'$. The reverse inclusion is rather obvious because it is rather clear that $\prep B^\circ$ contains $B$ and satisfies all the properties (\ref{item:regular}-\ref{item:limits}). 

So one can reformulate the non-trivial part of Corollary~\ref{cor:bipolar_of_operator} as follows~: if $T \notin B'$, then there is a Banach space $X$ such that $X \in \prep{B}$ but $\|T_X\|>1$. Constructing Banach spaces with prescribed properties is a notoriously difficult problem in general. What saves us here is that we do not construct $X$ explicitly, but we let the Hahn-Banach theorem construct it for us. This is achieved by suitably encoding the class of Banach spaces in a locally convex topological vector space $H$ and the class of operators between subspaces of $L_p$ spaces in its dual $H^*$, in such a way that the polarity between $\Banach$ and $\Op$ corresponds to the usual polarity in topological vector spaces. So, once these two encodings are well understood, both Theorems \ref{thm:hernandez} and \ref{thm:bipolar_of_operator_explicit} are just an application of the bipolar theorem in $H$ and $H^*$. When $X$ is a Banach space of dimension $n$, $X$ will be encoded inside the real Banach space $C(\KP^{n-1})$ of real-valued continuous functions on the projective space of dimension $n-1$. Similarly an operator $T$ with a domain of dimension $n$ will be encoded inside the dual of $C(\KP^{n-1})$. The space $H$ evoked would then be the projective limit of a suitable system of the spaces $C(\KP^{n-1})$. But since the study of the polarity between $\Banach$ and $\Op$ readily reduces to finite dimensional Banach spaces and operators with finite dimensional domains, we prefer to work directly with $C(\KP^{n-1})$ and never even formally introduce $H$.

\subsection*{Notation}

To avoid any set-theoretical problem (of $\Op$ not being a set), all the measure spaces appearing here will be standard measure spaces taken in some fixed set containing $[0,1]$ with the Lebesgue measure and that is stable by taking equivalent measures, measurable subsets with restriction of the measure, and finite direct sums. By direct sum of a finite sequence $(\Omega_1,m_1),\dots,(\Omega_n,m_n)$ we mean the space $(\Omega_1\cup\dots\cup \Omega_n,m_1\oplus \dots \oplus m_n)$ where $\Omega_1\cup\dots\cup \Omega_n$ is the disjoint union and the measure is $A \mapsto \sum_i m_i(A \cap \Omega_i)$. None of the results depend on the choice.

The $\ell_p$-direct sum of a finite family $T_1,\dots,T_n$ of operators from $\dom(T_i) \subset L_p(\Omega_i,m_i)$ to $\ran(T_i) \subset L_p(\Omega'_i,m'_i)$ is the operator $T_1 \oplus \dots \oplus T_n$ from $\dom(T_1) \oplus \dots \oplus \dom(T_n) \subset L_p(\Omega_1\cup \dots \cup \Omega_n,m_1\oplus \dots \oplus m_n)$ to $\ran(T_1) \oplus \dots \ran(T_n) \subset L_p(\Omega'_1\cup \dots \cup \Omega'_n,m'_1\oplus \dots \oplus m'_n)$.

An operator $T \in \Op$ is called regular if $\polar{T}{X}<\infty$ for every Banach space $X$, or equivalently if $\polar{T}{\ell_\infty}<\infty$. In that case the quantity $\polar{T}{\ell_\infty} = \sup_X \polar T X$ is called the regular norm of $T$ and denoted $\|T\|_r$. We will denote by $REG$ the set of operators $T \in \Op$ such that $\|T\|_r \leq 1$.

\subsection*{Organization of the paper} The first section presents some necessary background and some motivation for studying this polarity. Section \ref{sec:preliminaries} contains various preliminaries, including basic reminders on measure theory and on the linear bipolar theorem, as well as one result on which the rest will rely: Proposition~\ref{prop:bipolar_with_compactness}. Section \ref{sec:Hn} contains the proof of the main theorem. It starts by defining the encoding of spaces and operators in a linear duality, and then studies this encoding. Section~\ref{section:comparison} contains a discussion of variants of the duality presented in the introduction. In an appendix we present a new proof and a generalization, in the context of Section \ref{sec:Hn}, of Hardin's theorem \cite{MR611233}. Hardin's theorem appears as a direct corollary of the study of the invariant subspaces for some families of representations of $\GL_n(\K)$ on $C(\KP^{n-1})$.

Some of the results have been announced in the report \emph{Group actions on Banach spaces and a duality spaces/operators} \cite[pp 2304--2307]{MR4148141}.

\section{Background and motivation}\label{sec:motivation}

\subsection{Reminders on Banach space geometry}
If $n$ is an integer, the set $Q(n)$ of all $n$-dimensional normed space up to isometry, equipped with the Banach-Mazur distance
\[d(E,F) = \inf\{ \|u\| \|u^{-1}\| \mid u \colon E \to F \textrm{ linear invertible}\},\]
becomes a compact metric space, \emph{the Banach-Mazur compactum}. Beware that it is not $d$ but $\log d$ which is a distance in the usual way ($d$ is submultiplicative $d(E,G) \leq d(E,F) d(F,G)$ rather subadditive, and two isometric spaces are at Banach-Mazur distance $1$), but following the tradition we still call $d$ the Banach-Mazur distance.

We say that a Banach space $X$ is finitely representable in another Banach space $Y$ if for every finite-dimensional space $E \subset X$ and every $\varepsilon>0$ there is a subspace $F \subset Y$ of same dimension as $E$ such that $d(E,F)\leq 1+\varepsilon$. In other words, if for every $n$, the closure in $Q(n)$ of the space of $n$-dimensional subspaces of $X$ is contained in the same closure but for $Y$. This is equivalent to $X$ being isometrically a subspace of an ultraproduct of $Y$ \cite{zbMATH03640303}. 

More generally, we say that a Banach space $X$ is finitely representable in a class $B$ of Banach spaces if for every finite-dimensional space $E \subset X$ and every $\varepsilon>0$ there is a subspace $F$ of a space in $B$ of same dimension as $E$ such that $d(E,F)\leq 1+ \varepsilon$. We can therefore define \emph{a class of Banach spaces up to finite representability} as a collection $A_n$ of closed subsets of $Q(n)$ such that for every $n>m$, every $m$-dimensional subspace of every $E \in A_n$ belongs to $A_m$. In this representation, finite representability corresponds to inclusion.

The $\ell_p$-direct sum of a finite family $X_1,\dots,X_n$ of Banach spaces is the space $X_1 \oplus X_2 \oplus \dots \oplus X_n$ for the norm $\|(x_1,\dots,x_n)\| = (\|x_1\|^p + \dots+\|x_n\|^p)^{\frac 1 p}$.

\subsection{Motivation}
Estimating $\|T_X\|$ in terms of the properties of $T$ and the geometric properties of $X$ is a central aspect in the geometry of Banach spaces. Most natural geometric classes of Banach spaces are characterized in terms of such quantities, and most celebrated results can be expressed in the form ``$T$ belongs to the bipolar of $B$'' for specific $T$ and $B \subset \Op$. We list a few historical important examples for illustration. See \cite[Section 4]{MR2732331} for other examples.
\begin{itemize}
\item Hilbert spaces are characterized by the parallelogram inequality, \emph{i.e.} the property $\|T_X\|\leq 1$ where $T\colon \ell_2^2 \to \ell_2^2$ has matrix $\frac{1}{\sqrt 2}\begin{pmatrix} 1 & 1 \\ -1 & 1 \end{pmatrix}$.
\item \cite{MR730072,MR727340} A Banach space has the UMD property (for Unconditional Martingale Differences) if and only if the Hilbert transform $H \colon L_2(\R) \to L_2(\R)$ satisfies $\|H_X\|<\infty$.
\item Let $(\Omega,\mu)$ be a probability space and $\varepsilon_i\colon \Omega \to \{-1,1\}$, $i \in \N$ be iid centered (Bernoulli) random variables. A Banach space $X$ has type $p$ if there is a constant $T_p$ such that
  \[ \| \sum \varepsilon_i x_i\|_{L_p(\Omega;X)} \leq T_p (\sum \|x_i\|^p)^{\frac 1 p} \]
  for every $x_i \in X$. Equivalently if $\|T_X\| \leq 1 $ where $T \colon \vect(\varepsilon_i) \subset L_p(\Omega) \to \ell_p(\N)$ is the linear map sending $\varepsilon_i$ to $\frac{1}{T_p}(1_{k=i})_{k \in \N}$.

\item A Banach space $X$ has cotype $p$ if there is a constant $C_p$ such that
  \[ \| \sum \varepsilon_i x_i\|_{L_p(\Omega;X)} \geq \frac{1}{C_p} (\sum \|x_i\|^p)^{\frac 1 p} \]
  for every $x_i \in X$. Equivalently if $\|S_X\| \leq$ where $S \colon \ell^p \to  L_p(\Omega)$ is the linear map sending $(a_i)_{i \in \N}$ to $\frac{1}{C_p}\sum_i a_i \varepsilon_i$.
\item A Banach space $X$ has type $>1$ if and only if it is K-convex \cite{MR647811}: $\|T_X\|<\infty$, where $T \in B(L_2(\{-1,1\}^\N))$ is the orthogonal projection on the space spanned by the coordinates $\varepsilon_i \colon \omega = (\omega_n)_{n \in \N} \mapsto \omega_i$.
\item  Denote by $d_n(X)$ the supremum over all $n$-dimensional subspaces $E$ of $X$ subspaces of the Banach-Mazur distance from $U$ to $\ell_2^n$. Then (this is due to Pisier but written in \cite{dLdlSCrelle}) up to a factor $2$, $d_n(X)$ is equal to $\sup \|T_X\|$, where the sup is taken over all $T \colon L_2 \to L_2$ of norm $1$ and rank $n$. We can therefore express the Milman-Wolfson Theorem \cite{MR0467255} as follows: a Banach space $X$ has type $p>1$ if and only if $\|T_X\| = o( \|T\| \mathrm{rk}(T)^{\frac 1 2})$ as $\mathrm{rk}(T) \to \infty$.
\end{itemize}

We now move to a more detailed discussion of two of the author's main motivations.
\subsection{Group representations on Banach spaces}
Another motivation comes from the study of representations of groups on Banach spaces. Let $G$ be a locally compact topological group with a fixed left Haar measure. We recall that every strong-operator-topology (SOT) continuous representation $\pi$ of $G$ on a Banach space $X$ extends to a representation of the convolution algebra $C_c(G)$ of compactly supported continuous functions on $G$ by setting $\pi(f) x = \int f(g) \pi(g) x dg$ for every $x \in X$.

For example, if $\lambda_p$ denotes the left-regular representation on $L_p(G)$ $\lambda_p(g) f = f(g^{-1} \cdot)$, then $\lambda_p(f)$ is the convolution operator $\xi \mapsto f \ast \xi$.

When $A$ is a class of Banach spaces, denote by $C_A(G)$ the completion of $C_c(G)$ for the norm
\[ \|f\|_{C_A(G)} =  \sup \|\pi(f)\|_{B(X)},\] where the supremum is over all SOT-continuous continuous isometric representations $\pi$ of $G$ on a space $X$ in $A$.

The following result, which generalizes the classical fact that, for amenable groups, the full and reduced $C^*$-algebras coincide, reduces the understanding the representation theory of $G$ on a Banach space $X$ to the understanding of $\|T_X\|$ for convolution operators $T$. This known fact has already appeared in several unpublished texts (for example in the author's habilitation thesis), but seems to be missing from the published literature.
\begin{prop}\label{prop:amenable} If $G$ is amenable and $\pi$ is an isometric representation of $G$ on a Banach space $X$, then for every $f \in C_c(G)$,
  \[ \|\pi(f)\|_{B(X)} \leq \| \lambda_p(f)_X\|.\]
In particular, if a class of Banach spaces $A$ has the property that $L_p(G;X) \in A$ for every $X \in A$, then
\[ \|f\|_{C_A(G)} = \sup_{X \in A} \| \lambda_p(f)_X\|.\]
\end{prop}
\begin{proof} Fix a norm $1$ element $\xi \in L_p(G)$ and define an isometric linear map $\alpha \colon X \to L_p(G;X)$ by $\alpha(x)(g) = \xi(g) \pi(g^{-1}) x$.

Then for $h \in G$, $(\alpha( \pi(h) x) - \lambda(h) \alpha(x))(g) = (\xi(g) - \xi(h^{-1} g )) \pi(g^{-1} h)x$, and 
\[\|\alpha( \pi(h) x) - \lambda(h) \alpha(x) \| = \|x\| \|\xi - \lambda(h) \xi\|_{L_p(G)}.\]
By the triangle inequality
\[\|\alpha( \pi(f) x) - \lambda(f) \alpha(x) \| \leq \|f\|_{L_1(G)}\|x\| \sup_{h \in \mathrm{supp}(f)} \|\xi - \lambda(h) \xi\|_{L_p(G)},\]
and using that $\alpha$ is isometric we obtain
\[ \| \pi(f)x\|   \leq \|\lambda(f)_X\| \|x\| + \|f\|_{L_1(G)} \|x\|\sup_{h \in \mathrm{supp}(f)} \|\xi - \lambda(h) \xi\|_{L_p(G)}.\]
We deduce
\[ \| \pi(f)\| \leq \|\lambda(f)_X\| + \|f\|_{L_1(G)} \sup_{h \in \mathrm{supp}(f)} \|\xi - \lambda(h) \xi\|_{L_p(G)}.\]
When $G$ is amenable, the last term can be made arbitrarily small, which proves the proposition.
\end{proof}
In the particular case of a compact group, this result lies at the heart of the proofs of Lafforgue's strong property (T) for higher-rank algebraic groups. For example, thanks to the techniques of strong property (T), the conjecture \cite{MR2316269} that any action by isometries of a lattice in a connected higher-rank simple Lie group on a super-reflexive Banach space has been reduced to the following conjecture, see \cite{MR2423763,MR2574023,strongTsp4}, see also \cite{Salle2015}. Denote, for any $\delta \in [-1,1]$, by $T_\delta$ the operator on $L_2(\mathbf{S}^2)$ mapping $f$ to the fonction $(T_\delta f)(x) =$ the average of $f$ on the circle $\{y \in \mathbf{S}^2 \mid \langle x,y\rangle = \delta\}$. For any $\theta \in \R/2\pi\Z$, denote by $S_\theta$ the operator on $L_2(\mathbf{S}^3)$ mapping $f$ to the fonction $(S_\theta f)(z) =$ the average of $f$ on the circle $\{\frac{1}{\sqrt 2}(e^{i\theta} + e^{i\varphi}j)z \mid \varphi \in \R/2\pi \Z\}$ (where we identify $\mathbf{S}^3$ with the norm $1$ quaternions in the usual way). The conjecture is that for every super-reflexive Banach space, there exist $\alpha>0$ and $C \in \R_+$ such that for every $\delta \in [-1,1]$ and $\theta \in \R$,
\[ \|(T_\delta - T_0)_X \| \leq C |\delta|^\alpha \textrm{ and }\|(S_\theta - S_{\pi/2})_X\| \leq C |\theta - \pi/2|^\alpha.\]

\subsection{Super-expanders and embeddability of graphs in Banach spaces}
Another motivation for studying the quantity $\|T_X\|$ is its well-known connection with Poincar\'e inequalities and embeddability of expanders in $X$. If $\mathcal{G}=(V,E)$ is a finite connected graph, we may define\footnote{There are many small variants of the definition. But they do not matter for the discussion here, though they do matter for other issues, see for example \cite{L2spectralGap}.} its $X$-valued $p$-Poincar\'e constant $\pi_{p,\mathcal{G}}(X)$ as the smallest constant $\pi$ such that for every $f \colon V \to X$ satisfying $\sum_{v \in V} \mathrm{deg}(v) f(v) =0$,
\[ \left(\sum_{v \in V} \mathrm{deg}(v) \|f(v)\|^p\right)^{\frac 1 p} \leq \pi \left(\sum_{(v,w) \in E} \|f(v) - f(w)\|^p\right)^{\frac 1 p}.\]
Note that $\pi_{p,\mathcal{G}}(X) = \|T_X\|$ for $T$ the inverse of the linear map $f \in \ell_p^0(V,\mathrm{deg}) \mapsto (f(v) - f(w))_{(v,w) \in E} \in \ell_p(E)$.

A sequence $\mathcal{G}_n=(V_n,E_n)$ of bounded degree graphs is called \emph{a sequence of expanders with respect to $X$} if $\lim_n |V_n|=\infty$ and $\sup_n \pi_{p,\mathcal{G}_n}(X)<\infty$. This does not depend on $p$ \cite{MR3356758,MR3208067,MR3503071}, see also \cite[Proposition 3.9]{L2spectralGap}. 

For example, if $p=2$ and $X=\K$ (or a Hilbert space), then $\pi_{p,2}(\K)$ is equal to $(2-2\lambda_2)^{-\frac 1 2}$, for $\lambda_2$ the second largest eigenvalue of the random walk operator on $\mathcal{G}$. So being a sequence of expanders with respect to $\K$, or to an $L^p$ space for some $p<\infty$, is the same as the usual definition of expander graphs. 

According to \cite{MR3210176}, a sequence $\mathcal{G}_n$ is called a sequence of super-expanders if they are expanders with respect to all uniformly convex Banach spaces. The existence of super-expanders is a difficult result. Essentially two classes of examples have been obtained, by Lafforgue \cite{MR2574023} and by Mendel and Naor \cite{MR3210176}. Lafforgue's examples are even expanders with respect to all Banach spaces of type $>1$. All these results are therefore results of the norm ``$T$ belongs to be bipolar of $S$'', where $S$ is any of the operators quantifying the fact that a Banach space has nontrivial type or is super-reflexive, and $T$ are correctly scaled operators in the definition of the $p$-Poincar\'e constant. Many intriguing questions remain open, which can all be formulated in the same way. For example, 
\begin{que}\cite{MR3210176} Are all expander sequences super-expanders? Expanders with respect to all spaces of non-trivial type?
  \end{que}

\begin{que}\cite{MR3210176,MR2574023} Does there exist a sequence of super-expanders of girth going to infinity? And of logarithmic girth in the number of vertices? Are the expanders coming from higher-rank simple Lie groups super-expanders?
  \end{que}

\begin{que}\cite{MR3210176,MR2574023} Does there exist a sequence of expanders with respect to all Banach spaces of nontrivial coptype?
\end{que}
A positive answer to this question is conjectured in \cite{MR3210176},
and Lafforgue even suggests that the super-expanders coming from
lattices in $\mathrm{SL}_3(\Q_p)$ (or other higher-rank simple
algebraic groups over non-archimedean local fields) as in
\cite{MR2574023} are such examples. But this is wide open, as is the following.

\begin{que}\cite{MR2732331} Are all expander sequences expanders with respect to all spaces of non-trivial cotype?
\end{que}

One of the reasons for the interest in expanders with respect to Banach spaces is the well-known fact, which essentially goes back to Gromov, that a sequence of expanders with respect to $X$ does not coarsely embed into $X$. See for example \cite[Section 3]{MR2732331}. Being an expander with respect to $X$ is much stronger than non coarse embeddability (a striking example is given in \cite{MR3334230}), but by \cite{MR2649350} (see also \cite{zbMATH05566508} for $L_1$ spaces) there is equivalence between non-coarse embeddability into families of Banach spaces under closed finite representability and $\ell_p$ direct sums and some \emph{other} forms of Poincar\'e inequalities.

\section{Preliminaries}\label{sec:preliminaries}

\subsection{On the bipolar in a dual Banach space}\label{subsec:polar_linear}
In the whole paper, for a subset $C$ of a real Banach space $E$ with dual $E^*$, we denote its polar
\[ C^\circ = \{ x^* \in E^*, \langle x^*,x \rangle \geq -1\textrm{ for all }x \in C\}.\]
When $C \subset E$ is a cone (that is $x \in C$ implies $\{tx \mid t \in [0,\infty)\} \subset C$), then its polar $C^\circ$ coincides with $\{ x^* \in E^*, \langle x^*,x \rangle \geq 0\textrm{ for all }x \in C\}$. It is also a cone.

Similarily, when $C \subset E^*$ we denote its polar for the weak-* topology by 
\[ \prep{C} = \{x \in E, \langle x^*,x \rangle \geq -1\textrm{ for all }x^* \in C\}.\]
Again, if $C$ is a cone, $\prep{C}$ coincides with $\{x \in E, \langle x^*,x \rangle \geq 0\textrm{ for all }x^* \in C\}$ and is again a cone.

It should be always clear from the context whether the polarity is considered in this linear setting of two vector spaces in duality or between $\Banach$ and $\Op$ as in Definition~\ref{defn:polar} and \ref{defn:prepolar}.

The classical bipolar theorems in this setting take the following forms:
\begin{thm}\label{thm:bipolar_Banach} Let $E$ be a real Banach space.

  If $C \subset E$, then its bipolar $\prep{(C^\circ)}$ is equal to the norm closure of the convex hull of $C \cup\{0\}$.

  If $C \subset E^*$, then its bipolar $(\prep{C})^\circ$ is equal to the weak-* closure of the convex hull of $C \cup\{0\}$.
\end{thm}

The second statement is not so useful for our purposes because taking the weak-* closure can be quite complicated, as we shall soon recall. Fortunately, there is an interesting consequence of the Krein-Smulian theorem \cite[Theorem V.12.1]{MR1070713}, which asserts that a convex subset of $E^*$ for a separable Banach space $E$ is weak-* closed if and only if it is sequentially weak-* closed, see \cite[Theorem V.12.10]{MR1070713}. This allows to significantly strengthen the result for separable Banach spaces as follows.

If $C$ is a subset of a dual $E^*$, let us define an increasing family of subsets $C_\alpha \subset E^*$ indexed by the ordinals $\alpha$ by letting $C_0=C$, $C_\alpha$ be the set of all weak-* limits of sequences in $C_{\alpha-1}$ if $\alpha$ is a successor and $C_\alpha = \cup_{\beta <\alpha} C_\beta$ if $\alpha$ is a limit ordinal. The smallest ordinal $\alpha$ such that $C_\alpha = C_{\alpha+1}$ (that is $C_\alpha$ is weak-* sequentially closed) is sometimes called the order of $C$. When $E$ is separable, the order of $C$ is countable, see for example the argument in the proof of \cite[Theorem V.12.10]{MR1070713}. 
Moreover, if $C$ is convex, then so is $C_\alpha$ for every $\alpha$. It follows from  \cite[Theorem V.12.10]{MR1070713} that, for the order of $C$, $C_\alpha$ coincides with the weak-* closure of $C$. Let us summarize this discussion.

\begin{prop} Let $E$ be a real separable Banach space and $C$ be a subset of $E^*$. There is a countable ordinal $\alpha$ such that the bipolar of $C$ coincides with $(\operatorname{conv}(C))_\alpha$.\end{prop}

The smallest ordinal $\alpha$ in the previous proposition measures the difficulty to construct the bipolar of $C$ out of $C$.

The order has been more studied for linear subspaces $C$. It is known that for many cases, every countable ordinal appears as the order of a linear subspace of $E^*$. This was stated by Banach \cite{MR231182}, and later examples such as $E^*=\ell_1=(c_0)^*,\ell_\infty,H_\infty$ were provided together will full proofs \cite{MR234264,MR227742}. It is now known that this holds whenever $E$ is not quasi-reflexive, that is when the canonical image of $E$ has infinite codimension in its bidual \cite{zbMATH04038488}. See also the survey \cite{zbMATH01989630} for more information on this.

It turns out that, for our applications, the order will always be equal to $1$. This will follow from the following result.
\begin{prop}\label{prop:bipolar_with_compactness} Let $E$ be a real Banach space and $C \subset E^*$. Assume that there is a convex subset $A \subset E$ such that $A \cap \{x \in E \mid \|x\| \leq r\}$ is norm-compact for every $r>0$ and $A^\circ \subset C$.

  Then the bipolar $(\prep{C})^\circ$ of $C$ is equal to the norm closure of the convex hull of $C$.
\end{prop}
\begin{proof}
Note that our assumptions implies that $0 \in C$ (as $0 \in A^\circ$). Let $C'$ be the norm closure of the convex hull of $C$. We know from the bipolar theorem (Theorem \ref{thm:bipolar_Banach}) that $(\prep{C})^\circ$ is equal to the weak-* closure of $\mathrm{conv}(C)$, so the inlusion $C' \subset (\prep{C})^\circ$ is obvious. To prove the converse inclusion, consider $x \in E^* \setminus C'$. We have to prove that $x$ does not belong to the weak-* closure of the convex hull of $C$.

  Let $j \colon E \to E^{**}$ be the canonical inclusion of $E$ in its bidual. By the Hahn-Banach separation theorem in the Banach space $E^*$, there is $\varphi \in E^{**}$ such that $\inf_{C} \varphi \geq -1$ and $\varphi(x)<-1$. In particular, we have $\inf_{A^\circ} \varphi \geq -1$, that is $\varphi \in (A^\circ)^\circ = (\prep{(j(A))})^\circ$. By Theorem \ref{thm:bipolar_Banach} again, $(\prep{(j(A))})^\circ$ is equal to the weak-* closure of (the convex set) $j(A)$. But the assumption on $A$ implies that $j(A)$ is already weak*-closed. Indeed, by the Krein-Smulian theorem, it is enough to show that $j(A) \cap \overline{B}_{E^{**}}(0,r)$ is weak-* closed for every $r>0$. This is true as $j(A) \cap \overline{B}_{E^{**}}(0,r) = j(A \cap \overline{B}_{E}(0,r))$ is even norm-compact as a continuous image of a norm-compact set, and norm-compact subsets of $E^{**}$ are weak-* closed. So $j(A)$ being weak-* closed, we have proved that $\varphi \in j(A)$. In particular, $\varphi$ is $\sigma(E^*,E)$-continuous, and we obtain, as announced, that $x$ does not belong to the weak-* closure of the convex hull of $C$. 
\end{proof}




\subsection{Reminders on the Jordan decomposition of measures}
Recall that any signed measure $m$ on a Borel space has a unique decomposition $m = m_+-m_-$ for two positive measures satisfying $\|m\| = \|m_+\|+\|m_-\|$ (where the norm is the total variation norm). This is the Jordan decomposition of $m$. If $m=m_1-m_2$ is any other decomposition with $m_1,m_2$ positive measures, then $m_1-m_+=m_2-m_-$ is a positive measure. We will use the following elementary fact.
\begin{lem}\label{lem:jordan_decomposition} Let $m$ and $m'$ be any signed measure, and let $m_1,m_2$ be any positive finite measures such that $m=m_1-m_2$. There is a decomposition $m'=m'_1-m'_2$ with
  \[ \|m_1-m'_1\| + \|m_2-m'_2\| = \|m - m'\|.\]
\end{lem}
\begin{proof} Let $m=m_+-m_-$ and $m'=m'_+-m'_-$ be the Jordan decompositions. A small computation gives that $\|m-m'\| = \|m_+ - m'_+\| + \|m_--m'_-\|$.

  By the property of the Jordan decomposition just recalled, $m'':=m_1-m_+=m_2-m_-$ is a positive measure. Define $m'_1=m'_++m''$ and $m'_2=m'_2+m''$, so that $m'=m'_1-m'_2$ and
  \[ \|m_1-m'_1\| + \|m_2-m'_2\| = \|m_+-m'_+\| + \|m_--m'_-\| = \|m-m'\|.\]
  \end{proof}

\subsection{On (\ref{item:subtle}) in Corollary~\ref{cor:bipolar_of_operator}}
This short subsection is not needed anywhere else in the paper, but it hopefully illustrates some basic things about Theorem~\ref{thm:bipolar_of_operator_explicit} and Corollary~\ref{cor:bipolar_of_operator}. We start by a lemma which clarifies in which situation an operator $T$ is of the form (\ref{item:subtle}) in Corollary~\ref{cor:bipolar_of_operator}.
\begin{lem} Let $T$ be a norm $\leq 1$ operator between subspaces $\dom(T),\ran(T) \subset L_p(\Omega,m)$ and $A \subset \Omega$ measurable. The following are equivalent.
  \begin{itemize}
  \item $Tf(x) = f(x)$ for almost every $x \in \Omega \setminus A$ and every $f \in \dom(T)$.
  \item If we write $L_p(\Omega,m) = L_p(A,m) \oplus_p L_p(\Omega \setminus A,m)$, then there is an operator $S$ with domain equal to the image of $\dom(T)$ by the first coordinate projection such that $T(f_1,f_2) = (Sf_1,f_2)$ for all $(f_1,f_2) \in \dom(T)$.
  \end{itemize}
  In that case, $S$ is unique, $\dom(S) = \{ f\left|_{A}\right. , f \in \dom(S)\}$ and $S (f\left|_{A}\right.) = (Tf)\left|_{A}\right.$ for all $f \in \dom(T)$.
\end{lem}
\begin{proof} Clearly, the assumption that $Tf(x) = f(x)$ for almost every $x \in \Omega \setminus A$ and every $f \in \dom(T)$ is equivalent to the existence of a linear map $S \colon \dom(T) \to L_p(A,m)$ such that $T(f_1,f_2) = (S(f_1,f_2),f_2)$. So to prove the equivalence stated in the lemma, we have to observe that, in this situation, $S(f_1,f_2)$ depends only on $f_1$, \emph{i.e.} (by linearity) that $S(f_1,f_2)=0$ if $f_1=0$. For $(0,f_2) \in \dom(T)$ we have $\|T (0,f_2)\|_p^p = \|S(0,f_2)\|_p^p+\|f_2\|^p$, which (by the assumption that $\|T\|\leq 1$) is less than $\|f_2\|^p$. This proves that $\|S(0,f_2)\|_p^p =0$, as requested.

  The last assertion is a tautology.
\end{proof}

Finally, we provide an example that illustrates the main result.
\begin{ex}\label{example:composition} The inequality $\| (T\circ S)_X \| \leq \|T_X\| \|S_X\|$ is clear for every Banach space $X$ and every operators $T,S$ such that $T \circ S$ makes sense. So it follows from Corollary~\ref{cor:bipolar_of_operator} that, with the notation therein, if $S,T \in B$ then $T \circ S$ belongs to $B'$. We prove this directly, because it illustrates the subtle property (\ref{item:subtle}).

So let $S,T\in B$ such that $\ran(S) \subset \dom(T)$. By (\ref{item:ellpdirectsum}) the operator $S\oplus T \colon \dom(S) \oplus \dom(T) \to \ran(S) \oplus \ran(T)$ belongs to $B'$. By composing by the spatial isometry $(f,g) \in \ran(S) \oplus \ran(T) \mapsto (g,f) \in \ran(T) \oplus \ran(S)$ (which is allowed by (\ref{item:composing_by_spatial_isom}))  and restricting to the subspace $D=\{(f,Sf) | f \in \dom(S)\} \subset \dom(S) \oplus \dom(T)$ (which is allowed by (\ref{item:limits})), we obtain that the map $(f,Sf) \in D \mapsto (T \circ S f,Sf)$ belongs to $B'$. By (\ref{item:subtle}), we conclude that $T \circ S$ belongs to $B'$ as required. 
\end{ex}

\section{The space of degree $p$ homogeneous functions on $\K^n$}\label{sec:Hn}

Let $n$ be a positive integer. Denote by $|z|$ the $\ell_p$-norm on $\K^n$
\[ |z| = \left( |z_1|^p+ \dots + |z_n|^p \right)^{\frac 1 p}.\]
In the rare occasions when we want to insist on $p$, we write $|z|_p$ for this quantity.

 A function $\varphi \colon \K^{n} \to \R$ is called homogeneous of degree $p$ if $\varphi(\lambda z) = |\lambda|^p \varphi(z)$ for all $z \in \K^n$ and $\lambda \in \K$. The space $H_n$ of continuous homogeneous of degree $p$ functions on $\K^n$ is a Banach space over the field of real numbers for the topology of uniform convergence on compact subsets on $\K^n$. A particular choice of norm is $\|\varphi\|= \sup_{|z| \leq 1} |\varphi(z)|$, so that for this norm $H_n$ is isometrically isomorphic to the space of real-valued continuous functions on $\KP^{n-1}$ through the identification of $\varphi \in H_n$ with the function $\K z \in \KP^{n-1} \mapsto \varphi(\frac{z}{|z|})$. An equivalent definition of the norm of $\varphi \in H_n$ is the smallest number such that for every $z \in \K^n$
\begin{equation}\label{eq:def_norm_Hn} |\varphi(z)| \leq (|z_1|^p + \dots+|z_n|^p) \|\varphi\|.\end{equation}


We encode a class $A \subset \Banach$ of Banach spaces by the cone $N(A,n) \subset H_n$ (N for norms) of functions of the form $z \mapsto \| \sum_{i=1}^n z_i x_i\|^p$ for $X \in A$ and $x_1,\dots,x_n \in X$.

When $(\Omega,m)$ is a measure space and $f=(f_1,\dots,f_n)$ is an $n$-uple of elements of $L_p(\Omega,m)$, we can define a continuous linear form $\mu_f$ on $H_n$ by
\begin{equation}\label{eq:def_muf}\dual{\mu_f}{\varphi}= \int \varphi(f_1(\omega),\dots,f_n(\omega)) dm(\omega).\end{equation}
Indeed, it follows from \eqref{eq:def_norm_Hn} that the integral is well-defined and that $\mu_f \in H_n^*$ with norm equal to $\|f_1\|_p^p+ \dots+\|f_n\|_p^p$ (the inequality $\leq$ is immediate from \eqref{eq:def_norm_Hn}, and the equality follows by evaluating $\mu_f$ at the norm $1$ element $z \mapsto |z|^p$ in $H_n$).

We encode a class $B \subset \Op$ of operators by the cone $P(B,n) \subset H_n^*$
\[ P(B,n) = \{ \mu_f - \mu_{Tf}, T \in B \textrm{ and }f \in \dom(T)^n\}\]
where for $f =(f_1,\dots,f_n)\in \dom(T)^n$, we denote $Tf = (Tf_1,\dots,Tf_n)$. It is a cone because for every $t\geq 0$,  $t (\mu_f - \mu_{Tf}) = \mu_{t^{\frac 1 p}f} - \mu_{T t^{\frac 1 p}f}$. 

The crucial but obvious property motivating these definitions is that, if $\varphi(z) = \|\sum_{i=1}^n z_i x_i\|_X^p$ for elements $x_1,\dots,x_n$ in a Banach space $X$, then $\dual{\mu_f}{\varphi} = \|\sum_i f_i x_i\|_{L_p(\Omega,m;X)}^p$. As a consequence,
\[ \dual{\mu_f - \mu_{Tf}}{\varphi} = \|\sum_i f_i x_i\|_{L_p(\Omega,m;X)}^p - \|\sum_i (Tf_i) x_i\|_{L_p(\Omega,m;X)}^p.\]
In particular, we have
\begin{lem}\label{lem:polarity_basic} Let $A \subset \Banach$ be a class of Banach spaces and $B \subset \Op$ a class of operators. 
\begin{enumerate}
\item\label{item:Pofpolarbasic}$B \subset A^\circ$ if and only if for every $n$, $P(B,n) \subset N(A,n)^\circ$. 
\item\label{item:Nofpolarbasic} $A \subset \prep{B}$ if and only if for every $n$, $N(A,n) \subset \prep{P(B,n)}$.
\end{enumerate}
\end{lem}
\subsection{Polarity in $H_n$}
We start by improving Lemma~\ref{lem:polarity_basic}. The next result expresses that the polarity in $\dual \Banach \Op$ (see Definition~\ref{defn:polar} and \ref{defn:prepolar}) is well encoded by the polarity $\dual{H_n}{H_n^*}$ (see Subsection~\ref{subsec:polar_linear}). Recall that $REG$ the class of all operators $T \in \Op$ with regular norm $\|T\|_r:=\sup_{X\in \Banach} \|T_X\| \leq 1$. 

\begin{prop}\label{prop:polarity} Let $A \subset \Banach$ be a class of Banach spaces and $B \subset \Op$ a class of operators. Then
\begin{enumerate}
\item\label{item:Pofpolar} $P(A^\circ,n) = N(A,n)^\circ$.
\item\label{item:Nofpolar} $N(\prep{B},n) \subset \prep{P(B \cup REG,n)}$.
\end{enumerate}
\end{prop}

In the proof, we need a description of the dual of $H_n$~:
\begin{lem}\label{lem:dual_Hn} Every continuous linear form $l$ on $H_n$ is of the form $\mu_f-\mu_g$ for some measure spaces $(\Omega,m)$ and $(\Omega',m')$ and $n$-uples $f \in L_p(\Omega,m)^n$ and $g \in L_p(\Omega',m')^n$. Moreover $\Omega,m,f$ and $\Omega',m',g$ can be chosen so that $f$ and $g$ take almost surely their values in $\{z \in \K^n, |z|=1\}$ and so that $m(\Omega)+m'(\Omega')$ is equal to the norm of $l$.
\end{lem}
\begin{proof}
By the identification of $H_n$ with $C(\KP^{n-1})$ and by the Riesz representation theorem, every continuous linear form $l$ on $H_n$ is of the form
\[ \varphi \mapsto \int_{\KP^{n-1}} \varphi\left(\frac{z}{|z|}\right) d\nu(\K z)\]
for a unique signed measure $\nu$ on $\KP^{n-1}$, and the norm of $l$ is the total variation of $\nu$. Let $\nu=\nu_+ - \nu_-$ be the Jordan decomposition of $\nu$ and $s\colon \KP^{n-1} \to \{z \in \K^{n},|z|=1\}$ a measurable section. Define $(\Omega,m) = (\KP^{n-1},\nu_+)$ and $f \in L_p(\Omega,m)^n$ by $s(\omega) = (f_1(\omega),\dots,f_n(\omega))$. Similarly define $(\Omega',m') = (\KP^{n-1},\nu_-)$ and $g \in L_p(\Omega',m')^n$ by $s(\omega) = (g_1(\omega),\dots,g_n(\omega))$. Then we have 
\[ \int_{\KP^{n-1}} \varphi\left(\frac{z}{|z|}\right) d\nu(\K z) 
= \dual{\mu_f-\mu_g}{\varphi}.\] This proves the lemma, because by construction $f,g$ both take values in $\{z \in \K^n, |z|=1\}$ and $m(\Omega) + m'(\Omega) = (\nu_+ + \nu_-)(\KP^{n-1})$ is the norm of $l$.
\end{proof}

\begin{proof}[Proof of Proposition \ref{prop:polarity}]
We start by (\ref{item:Pofpolar}). If every space in $A$ is trivial (of dimension $0$), we have $N(A,n)^\circ = H_n^*$, $A^\circ=\Op$, and the result is easy. We can therefore assume that $A$ contains a space of dimension $\geq 1$. Let $f,g$ be $n$-uples in $L_p$ spaces. Note that if $\varphi(z) = \|\sum_{i=1}^n z_i x_i\|^p$ then 
\[ \dual{\mu_f-\mu_g}{\varphi} = \| \sum_i f_i x_i\|_{L_p(X)}^p - \| \sum_i g_i x_i\|_{L_p(X)}^p.\] 
So the linear form $\mu_f - \mu_g \in H_n^*$ belongs to $N(A,n)^\circ$ if and only if for every $X \in A$ and $x_1,\dots,x_n \in X$, $\|\sum f_i x_i \|_{L_p(X)}^p \geq \|\sum g_i x_i\|_{L_p(X)}^p$. Using that there is a nonzero $X \in A$, this holds if and only if there is a linear map $T$ sending $f_i$ to $g_i$ such that $T \in A^\circ$. This shows that $\mu_f - \mu_g$ belongs to $N(A,n)^\circ$ if and only if it belongs to $P(A^\circ,n)$. By Lemma \ref{lem:dual_Hn} every element of $H_n^*$ is of this form, which proves (\ref{item:Pofpolar}).

We move to (\ref{item:Nofpolar}). Denote by $C_n$ the closed convex cone $ C_n=N(\Banach,n)$. We first prove that $N(\prep{B},n) = \prep{P(B,n)} \cap C_n$. By definition $N(\prep{B},n) \subset C_n$. So we have to prove that for $\varphi \in C_n$, $\varphi \in N(\prep{B},n)$ if and only if $\varphi \in \prep{P(B,n)}$. But if $\varphi(z) = \|\sum_{i=1}^n z_i x_i\|^p$ and $X=\vect(x_1,\dots,x_n)$, then we have that $\varphi \in N(B^\circ,n)$ if and only if $\|T\otimes id_X\|\leq 1$ for all $T \in B$, if and only if for all $T \in B$ and $f_1,\dots,f_n \in \dom(T)$, $\|\sum_i T f_i x_i\|^p \leq \|\sum_i f_i x_i\|^p$, if and only if $\varphi \in P(B,n)^\circ$.

We can now conclude with  (\ref{item:Nofpolar}). By  (\ref{item:Pofpolar}) for $A = \Banach$, we have $C_n^\circ = P(REG,n)$. On the other hand, since $C_n$ is a closed convex cone, the bipolar theorem implies that $C_n = \prep{C_n^\circ}$, and hence $C_n = \prep{P(REG,n)}$. We therefore get
\begin{align*}N(\prep{B},n) &= \prep{P(B,n)} \cap \prep{P(REG,n)}\\
  & = \prep{(P(B,n) \cup P(REG,n))} \\
  & = \prep{P(B \cup REG,n)}.
\end{align*}
This proves (\ref{item:Nofpolar}).
\end{proof}
By the bipolar theorem in $H_n$ and $H_n^*$, we obtain
\begin{cor}\label{cor:bipolars}
Let $A \subset \Banach$ be a class of Banach spaces and $B \subset \Op$ a class of operators. Then 
\begin{enumerate}
\item $N(\prep{A}^\circ,n) = \overline{\mathrm{conv}}N(A,n)$.
\item $P(\prep B^\circ,n) = \overline{\mathrm{conv}}^{w*} P(B \cup REG,n)$.
\end{enumerate}
\end{cor}

The rest of this section consists in understanding the closed convex hulls of $N(A,n)$ and $P(B,n)$.

\subsection{Understanding the encoding of Banach spaces in $H_n$}
The following easy fact will be important later.
\begin{lem}\label{lem:compactnessNAn}For every integer $n$, bounded subsets of $N(\Banach,n)$ are relatively norm-compact.
\end{lem}
\begin{proof} By the Arzel\`a-Ascoli theorem, we have to prove that bounded subsets of $N(\Banach,n)$ are equicontinuous, seen in $C(\KP^{n-1})$. This follows from the triangle inequality. For example for $p=1$ and $\varphi(z) = \|\sum_i z_i x_i\|$, then we have
  \[ |\varphi(z) - \varphi(z')| \leq \| \sum_i (z_i - z'_i) x_i\| \leq \sum_i |z_i-z'_i| \varphi(e_i).\]
The case of arbitrary $p$ is similar. Alternatively, it follows from the case $p=1$ by continuity of the map $t \mapsto t^p$.
\end{proof}
Lemma~\ref{lem:compactnessNAn} allows to considerably strengthen the second statement in Corollary~\ref{cor:bipolars}, replacing weak-* closure by norm closure.
\begin{cor}\label{cor:bipolarsbis}
  Let $B \subset \Op$ a class of operators. Then
  \[P(\prep B^\circ,n) = \overline{\mathrm{conv}}^{\|\cdot\|} P(B \cup REG,n).\]
\end{cor}
\begin{proof} The set $N(\Banach,n)$ is a closed convex cone in $H_n$, so by Lemma~\ref{lem:compactnessNAn} $N(\Banach,n) \cap \{\varphi \in H_n \mid \|\varphi\|\leq r\}$ is norm-compact for every $r$. Moreover, we have that $N(\Banach,n)^\circ = P(REG,n)$ by Proposition~\ref{prop:polarity}. So, since $P(B \cup REG,n)$ contains $N(\Banach,n)^\circ$, Proposition~\ref{prop:bipolar_with_compactness} implies that its bipolar is equal to the norm closure of its convex hull.
  \end{proof}

Let us list elementary properties of $N$.
\begin{lem}\label{lem:understandN} Let $A,A_1,A_2 \subset \Banach$ be classes of Banach spaces. 

\begin{enumerate}
\item\label{item:Ncontainement} $N(A_1,n) \subset N(A_2,n)$ if and only if, for every $X \in A_1$, every subspace of dimension $\leq n$ of $X$ is isometric to a subspace of a space in $A_2$. 
\item The convex hull of $N(A,n)$ is equal to $N(\oplus_{\ell_p} A,n)$, where $\oplus_{\ell_p}A$ denotes the set of all finite $\ell_p$-direct sums of Banach spaces in $A$.
\item\label{item:closure} The norm closure of $N(A,n)$ in $H_n$ coincides with $N(\overline A,n)$ where $\overline A$ denotes the set of Banach spaces finitely represented in $A$.
\end{enumerate}
\end{lem}
As a consequence of (\ref{item:Ncontainement}) and (\ref{item:closure}), if two classes of Banach spaces $A_1, A_2$ are closed under finite representability, then $A_1=A_2$ if and only if $N(A_1,n) = N(A_2,n)$ for all $n$.
\begin{proof}
The first point is obvious from the following observation~: if $x_1,\dots,x_n$ (respectively $y_1,\dots,y_n$) are elements in a Banach space $X$ (respectively in a Banach space $Y$), then the functions $z \mapsto \| \sum_{i=1}^n z_i x_i\|^p$ and $z \mapsto \| \sum_{i=1}^n z_i y_i\|^p$ coincide if and only if there is an isometry from the linear span of $\{x_1,\dots,x_n\}$ to the linear span of $\{y_1,\dots,y_n\}$ sending $x_i$ to $y_i$.

If $\varphi_1,\dots,\varphi_k \in N(A,n)$ are given by $\varphi_j(z) = \| \sum_{i=1}^n z_i x_{i}^{(j)}\|_{X_j}^p$ then by the definition of the $\ell_p$-direct sum $X_1 \oplus_p \dots \oplus_p X_k$ we can write
\[ \sum_{j=1}^k \varphi_j(z) = \| \sum_{i=1}^n z_i (x_{i}^{(j)})_{1 \leq j \leq k} \|_{X_1 \oplus_p \dots \oplus_p X_k}^p.\]
This shows that $N(\oplus_{\ell_p} A,n)$ coincides with \[\{\varphi_1+ \dots+\varphi_k, k \in \N, \varphi_j \in N(A,n)\}.\] This is the convex hull of $N(A,n)$ because $N(A,n)$ is a cone.

We move to (\ref{item:closure}). If a sequence $\varphi_k \in N(A,n)$ converges uniformly on compact subsets to $\varphi \in H_n$, then $\varphi^{\frac 1 p}$  is the uniform limit on compact sets of the seminorms $\varphi_k^{\frac  1 p }$, so it is a seminorm on $\K^n$. This means that there is a Banach space $X \in \Banach$ and $x_1,\dots,x_n$ spanning $X$ such that $\varphi(z) = \| \sum_{i=1}^n z_i x_i\|^p$. The family $x_1,\dots,x_n$ might not be linearly independant, so we extract from it a basis of $X$. Without loss of generality we can assume that this basis is $x_1,\dots,x_m$ for some $m \leq n$. Write 
\[\varphi_k(z) = \| \sum_{i=1}^n z_i x_i^{(k)}\|_{X_k}^p\] 
for some $X_k \in A$ and $x_1^{(k)},\dots,x_n^{(k)}  \in X_k$. From the assumption that $\varphi_k$ converges uniformly on compacta to $\varphi$ and the assumption that $x_1,\dots,x_m$ is linearly independant, we get that for every $\varepsilon>0$ there is $k$ such that 
\[(1-\varepsilon) \varphi(z,0) \leq \varphi_k(z,0) \leq (1+\varepsilon) \varphi(z,0)\] for all $z \in \K^m$. This means that the linear map $u \colon X \to X_k$ sending $x_i$ to $(1-\varepsilon)^{-\frac 1 p} x_i^{(k)}$ for $i \leq m$ satisfies \[ \|x\| \leq \|u(x)\|\leq (\frac{1+\varepsilon}{1-\varepsilon})^{\frac 1 p} \|x\| \textrm{ for all }x \in X.\] Since $\varepsilon>0$ was arbitrary we have proved that $X$ is finitely representable in $A$, \emph{i.e.} that $\varphi \in N(\overline A,n)$. This proves that $\overline{N(A,n)}\subset N(\overline A,n)$. The converse inclusion is proved by reading the preceding argument backwards.
\end{proof}

We can conclude our proof of Hernandez' theorem.
\begin{proof}[Proof of Theorem \ref{thm:hernandez}] Let $A'$ be the class of Banach spaces which are finitely representable in the class of $\ell_p$-direct sums of spaces in $A$. It follows from Corollary \ref{cor:bipolars} and Lemma \ref{lem:understandN} that for every integer $n$,
\[ N(\prep{A}^\circ,n) = N(A',n).\]
Since both $\prep{A}^\circ$ and $A'$ are closed under finite representability, we get the equality $\prep{A}^\circ = A'$ by the remark following Lemma  \ref{lem:understandN}.
\end{proof}

\subsection{Proof of Theorem \ref{thm:hernandez_isom}}

If $X$ is at Banach-Mazur $\leq C$ from $\prep{A}^\circ$; then the inequality $\|T_X\| \leq C$ for every $T\in A^\circ$ is clear.

For the converse, we will need the following consequence of the Hahn-Banach theorem.
\begin{lem}\label{lem:polar_C(K)} Let $K$ be a compact Hausdorff topological space, and $C(K)$ the space of real-valued continuous functions on $K$. Let $A$ be a closed convex cone in the positive cone of $C(K)$ such that $A \cap B(0,1)$ is compact. Let $s\geq 1$. Then for every $\psi \in C(K)$, the following are equivalent
  \begin{itemize}
  \item $\exists \varphi \in A, \psi \leq \varphi \leq s \psi$.
  \item  $\dual{s \mu - \nu}{\psi} \geq 0$ for every positive measures $\mu,\nu$ on $K$ such that $\dual{\mu-\nu}{\varphi} \geq 0$ for all $\varphi \in A$.
    \end{itemize}
  \end{lem}
\begin{proof} $\implies$ is easy because the inequality $\psi \leq \varphi \leq s \psi$ implies $\dual{s \mu - \nu}{\psi} \geq \dual{\mu-\nu}{\varphi}$.

  For the converse, since $A$ is a convex cone, the set $B$ of $\psi$ satisfying $\exists \varphi \in A, \psi \leq \varphi \leq s \psi$ is a convex cone. Moreover, the compactness assumption on $A$ implies that $B$ is also closed. Assume that $\psi \notin B$. By Hahn-Banach there is a linear form on $C(K)$ which is nonnegative on $B$ and negative at $\psi$. By the Riesz representation theorem and the Hahn decomposition, this linear form can be written as $f \mapsto \int fd(\mu-\nu)$ for positive measures $\mu,\nu$ such that there is a Baire measurable subset $E \subset K$ satisfying $\nu(E)=0$ and $\mu(K \setminus E)=0$.
 
  Let $\varphi \in A$. Let $f_n\colon K \to [0,1]$ be a sequence of continuous functions converging in $L^1(K,\mu+\nu)$ to the indicator function of $E$. Then for every $n$, the function $(\frac 1 s+(1-\frac 1 s) f_n)\varphi$ belongs to $B$ so $\dual{\mu-\nu}{(\frac 1 s+(1-\frac 1 s) f_n)\varphi} > 0$. By making $n \to \infty$ we get $\dual{\mu-\nu}{\frac 1 s \varphi 1_{K \setminus E} + \varphi 1_E}\geq 0$, which can be written as $\dual{\frac 1 s \mu-\nu}{\varphi} \geq 0$.

  So we have $\dual{\frac 1 s \mu-\nu}{\varphi} \geq 0$ for every $\varphi \in A$, whereas $\dual{\mu-\nu}{\psi} < 0$. This proves the lemma.
\end{proof}

We can now prove the converse implication in Theorem \ref{thm:hernandez_isom}. Assume that $\|T_X\| \leq C$ for every $T\in A^\circ$. Let $x_1,\dots,x_n \in X$. Define $\psi \in H_n$ by $\psi(z) = \|\sum_{i=1}^n z_i x_i\|^p$, and view $\psi$ in $C(\KP^{n-1})$. The assumption that $\|T_X\| \leq C$ for every $T\in A^\circ$ implies that $\dual{C^p \mu - \nu}{\psi} \geq 0$ for every positive measures $\mu,\nu$ on $\KP^{n-1}$ such that $\dual{\mu-\nu}{\varphi} \geq 0$ for all $\varphi \in A$. By Lemma \ref{lem:polar_C(K)} (remember Lemma~\ref{lem:compactnessNAn}) this implies that there is $\varphi$ in the closed convex hull of $N(A,n)$ such that $\psi \leq \varphi \leq  C^p \psi$. By the proof of Theorem \ref{thm:hernandez}, there is a space $Y \in \prep{A}^\circ$ and $y_1,\dots,y_n \in Y$ such that $\varphi(z) = \|\sum_i z_i y_i\|^p$. By taking the $1/p$-th power in the inequality $\psi \leq \varphi \leq  C^p \psi$ we get that $\|\sum z_i x_i\| \leq \|\sum z_i y_i\| \leq C \|\sum z_i x_i\|$ for every $y \in \K^n$. This means that the linear span of $x_1,\dots,x_n$ is at Banach-Mazur distance $\leq C$ from the linear span on $\{y_1,\dots,y_n\}$ and concludes the proof.

\subsection{Understanding the encoding of operators in $H_n^*$}

\begin{lem}\label{lem:when_mu-mu=mu-mu} Let $f,g,\tilde f, \tilde g$ be $n$-uples of elements of $L_p$ spaces. Then $\mu_f-\mu_g = \mu_{\tilde f} - \mu_{\tilde g}$ if and only if there is $h \in L_p(\Omega,m)^n, \tilde h \in L_p(\tilde \Omega,\tilde m)^n$ such that $\mu_{(f_i \oplus h_i)_{i=1}^n} = \mu_{(\tilde f_i \oplus \tilde h_i)_{i=1}^n} $ and $\mu_{(g_i \oplus h_i)_{i=1}^n} = \mu_{(\tilde g_i \oplus \tilde h_i)_{i=1}^n}$.
\end{lem}
\begin{proof}
The if direction is easy, because $\mu_{(f_i \oplus h_i)_{i=1}^n} = \mu_f+\mu_h$.

For the converse, assume that $\mu_f-\mu_g = \mu_{\tilde f} - \mu_{\tilde g}$. Let $\nu_f$ be the positive measure on $\KP^{n-1}$ such that, for every $\varphi \in H_n$
\begin{equation}\label{eq:mu_nu} \dual{\mu_f}{\varphi} = \int_{\KP^{n-1}} \varphi\left(\frac{z}{|z|}\right) d\nu_f(\K z).
\end{equation}
Define similarly $\nu_g,\nu_{\tilde f},\nu_{\tilde g}$. Then $\nu_f-\nu_g = \nu_{\tilde f} - \nu_{\tilde g}$ is a signed measure on $\KP^{n-1}$. Let $\nu_+-\nu_-$ be its Jordan decomposition. By the properties of the Jordan decomposition, $\nu_f - \nu_+ = \nu_g - \nu_-$ is a positive measure on $\KP^{n-1}$, and therefore by the proof of Lemma \ref{lem:dual_Hn} it is of the form to $\nu_{\tilde h}$ for some $n$-uple  $\tilde h \in L_p(\tilde \Omega,\tilde m)$. Similarly, there is a $h \in L_p(\Omega,m)^n$ such that $\nu_{\tilde f} - \nu_+ = \nu_{\tilde g} - \nu_- = \nu_h$. 

We can rewrite these equalities as 
\[ \nu_+ = \nu_{f} - \nu_{\tilde h} = \nu_{\tilde f} - \nu_{h}\]
and
\[ \nu_- = \nu_{f} - \nu_{\tilde h} = \nu_{\tilde f} - \nu_{h}.\]
This implies that $\mu_f+\mu_h = \mu_{\tilde f} + \mu_{\tilde h}$ and $\mu_g+\mu_h = \mu_{\tilde g} + \mu_{\tilde h}$ and proves the lemma.
\end{proof}

\begin{lem}\label{lem:when_muf=mug} For two families $f\in L_p(\Omega,m)^n$ and $g\in L_p(\Omega',m')^n$, $\mu_f=\mu_g$  if and only if there is a spatial isometry $\vect\{f_1,\dots,f_n\} \to \vect\{g_1,\dots,g_n\}$ sending $f_i$ to $g_i$.
\end{lem}
\begin{proof}
The if direction is easy~: firstly if there is a measurable function $h \colon \Omega \to \K\setminus\{0\}$, if $(\Omega',m')=(\Omega,|h|^{-p}m)$ and $g_i=h f_i$ for all $i$, then for every $\varphi \in H_n$, $\varphi(g_1,\dots,g_n) = |h|^p\varphi(f_1,\dots,f_n)$ and therefore $\dual{\mu_g}{\varphi}=\dual{\mu_f}{\varphi}$. Secondly if $f_1,\dots,f_n$ and $g_1,\dots,g_n$ are equimeasurable outside of $0$ in the sense of Definition \ref{defn:spatial_isometry}, then $\int \varphi(f_1,\dots,f_n) dm= \int \varphi(g_1,\dots,g_n) dm'$ for every Borel function $\varphi$ vanishing at $0$ and such that the integrals are defined. In particular $\mu_f=\mu_g$.

For the converse, assume that $\mu_f=\mu_g$. Take a measurable section $s \colon \KP^{n-1} \to \K^n$ with values in $\{z \in \K^n, |z|=1\}$. Then there are measurable nonvanishing functions $h \colon \Omega\to \K^*$ and $h'\colon \Omega' \to \K^*$ such that $f(\omega) = h(\omega) s(\K f(\omega))$ for every $\omega \in \Omega$ such that $f(\omega) \neq 0$, and similarly $g(\omega') = h'(\omega') s(\K g(\omega'))$ if $g(\omega') \neq 0$. By replacing $m$ by $|h|^{-1/p} m$ and $f_i$ by $h_i f_i$ and similarly for $g$ we can assume that $h=1$ and $h'=1$, and we shall prove that $f$ and $g$ are equimeasurable outside of $0$. By this we mean that for every Borel $E \subset \K^n\setminus\{0\}$,  $m(\{\omega, (f_1(\omega),\dots,f_n(\omega)) \in E\}) = m'(\{\omega', (g_1(\omega'),\dots,T g_n(\omega')) \in E\})$. It is clear that this implies that, for every matrix $A \in M_{m,n}(\K)$, $Af$ and $Ag$ are equimeasurable outside of $0$, and therefore that the linear map sending $f_i$ to $g_i$ is well-defined and is as in the definition of equimeasurability outside of $0$.

By the identification of $H_n$ with $C(\KP^{n-1})$, using that $|f| \in \{0,1\}$ we have 
\[ \int_{\Omega \setminus f^{-1}(0)} \psi(\K f) dm = \int_{\Omega' \setminus g^{-1}(0)} \psi(\K g) dm'\]
for every continuous function $\psi \colon  \KP^{n-1} \to  \K$, and therefore also for every bounded Borel function $\varphi \colon \KP^{n-1} \to \K$. This implies, since $f$ and $g$ take values in $\{0\} \cup s(\KP^{n-1})$, that $\int \varphi(f) dm = \int \varphi(g) dm'$ for every Borel function $\K^{n} \to \K$ vanishing at $0$. Equivalently, $f$ and $g$ are equimeasurable outside of $0$.
\end{proof}
\begin{rem}\label{rem:when_muf=mug} The same proof shows actually a bit more~: if a linear map $T \colon E \subset L_p(\Omega,m) \to L_p(\Omega',m')$ satisfies $\mu_f = \mu_{Tf}$ for every $n$ and every $f \in E^n$, then $T$ is a spatial isometry. Indeed, since by our standing assumption $E$ (as every other space considered in this paper) is separable, we can find a sequence $(f_i)_{i \geq 0}$ generating a dense subspace of $E$ and satisfying $\sum_i \|f_i\|^p<\infty$, and in particular $(f_i(\omega))_{i \geq 0}$ belongs to $\ell_p$ for almost every $\omega$. Then the same proof applies, except that we replace $\K^n$ by $\ell_p$ and $\KP^{n-1}$ by its projectivization $\ell_p/\K^*$.
\end{rem}
We shall also need the following variant~:
\begin{lem}\label{lem:when_muf_and_mug_are_close} For two families $f\in L_p(\Omega,m)^n$ and $g\in L_p(\Omega',m')^n$ and $\varepsilon>0$, $\|\mu_f - \mu_g\| < \varepsilon$ if and only if there are spatial isometries $U \colon \vect\{f_1,\dots,f_n\} \to L_p(\Omega'',m'')$ and $V \colon \vect\{g_1,\dots,g_n\} \to L_p(\Omega'',m'')$ such that
  \[  \int_{\Omega''} (|Uf|^p + |Vg|^p) \chi_{Uf \neq Vg} < \varepsilon.\]
\end{lem}
\begin{proof}
We prove the slightly stronger statement with $<\varepsilon$ replaced by $\leq \varepsilon$.

  The if direction is easy~: by Lemma~\ref{lem:when_muf=mug} we have $\mu_f= \mu_{Uf}$ and $\mu_g = \mu_{Vg}$, and therefore for every $\varphi \in H_n$,
  \begin{align*} \langle \mu_f-\mu_g,\varphi\rangle & = \int_{\Omega''} \varphi(Uf) - \varphi(g)\\
    & \leq \int_{\Omega''} (|\varphi(Uf)|+|\varphi(Vg)|) \chi_{f \neq Ug}\\
    & \leq \int_{\Omega''} (|Uf|^p + |Vg|^p) \|\varphi\| \leq \varepsilon \|\varphi\|.
  \end{align*}
Taking the supremum over $\varphi$ we get $\|\mu_f - \mu_g\|\leq \varepsilon$.

The converse follows from a coupling argument. Assume that $\|\mu_f-\mu_g\|\leq \varepsilon$. Let $\nu_f$ and $\nu_g$ be the measures on $\KP^{n-1}$ given by \eqref{eq:mu_nu}, so that the total variation norm of $\nu_f-\nu_g$ is at most $\varepsilon$. This means that we can decompose $\nu_f = \nu_0+\nu_1$ and $\nu_g = \nu_0 + \nu_2$ for positive measures with $(\nu_1 + \nu_2)(\KP^{n-1}) \leq \varepsilon$. As in the proof of Lemma~\ref{lem:dual_Hn}, each $\nu_k$ corresponds by  \eqref{eq:mu_nu} to $\mu_{h^k}$ for an $n$-uple $h^k \in L_p(\Omega_k,m_k)^n$ with $\sum_i \|h^k_i\|_p^p = \nu_k(\KP^{n-1})$. In particular, we have $\mu_f = \mu_{h^0} + \mu_{h^1}$ and $\mu_g = \mu_{h^0} + \mu_{h^2}$.

Let us define $\Omega''$ as the disjoint union $\Omega_0 \cup \Omega_1 \cup \Omega_2$, $m''$ as $m_0+m_1+m_2$, and $f'=h^0\oplus h^1 \oplus 0$ and $g' = h^0 \oplus 0 \oplus h^1$, so that $\mu_{f'} = \mu_{h^0}+\mu_{h^1} = \mu_f$ and $\mu_{g'} = \mu_g$. By Lemma~\ref{lem:when_muf=mug}, there are spatial isometries $U$ and $V$ sending $f$ to $f'$ and $g$ to $g'$ respectively, and we have
\[ \int_{\Omega''} (|f'|^p + |g'|^p)\chi_{f' \neq g'} = \int_{\Omega_1} |h^1|^p + \int_{\Omega_2} |h^2|^p \leq \varepsilon.\]
This proves the lemma.
\end{proof}
There is also an asymetric variant of the preceding lemma, that can be useful.
\begin{rem}\label{rem:when_muf_and_mug_are_close}  In Lemma~\ref{lem:when_muf_and_mug_are_close}, we can moreover assume that $(\Omega'',m'') = (\Omega\times [0,1], m \otimes d\lambda)$ (for $\lambda$ the Lebesgue measure), and that the spatial isometry $U$ is simply $U\xi(\omega,s) =\xi(\omega)$.
\end{rem}
\begin{proof}
Let $\mu_f,\mu_g,\nu_f=\nu_0+\nu_1,\nu_g=\nu_0+\nu_2$ be as in the proof of Lemma~\ref{lem:when_muf_and_mug_are_close}, where $\|\nu_1+\nu_2\|<\varepsilon$. We can even assume that $\nu_1 \neq 0$ (this is where the strict inequality $<\varepsilon$ is used). Denote by $\frac{d\nu_0}{d \nu_f} \colon \KP^{n-1} \to [0,1]$ the Radon-Nikodym derivative. Define $A \subset \Omega \times [0,1]= \Omega''$ by $A = \{(\omega,s) \mid s \leq 0 \leq s \leq \frac{d\nu_0}{d \nu_f}(\K^* f(x))\}$, so that $\mu_{f \chi_A} = \nu_0$ and $\mu_{f \chi_{\Omega''\setminus A}} = \nu_1$. In particular, $\Omega''\setminus A$ has positive measure and is therefore an atomless standard measure space, and we can find $h \in L_p(\Omega'',m'')^n$ that vanishes on $A$ such that $\mu_h$ corresponds to $\nu_2$. We then have $\mu_g = \mu_h + \mu_{f \chi_A} = \mu_{h+f\chi_A}$. The last equality is because $h$ and $f\chi_A$ are disjointly supported. By Lemma~\ref{lem:when_muf=mug}, there is a spatial isometry $V$ sending $g$ to $h+f\chi_A$. Moreover, we have
  \[ \int_{\Omega''} (|f|^p + |h+f\chi_A|^p)\chi_{f \neq h + f \chi_A} \leq \int_{\Omega''\setminus A} (|f|^p + |h|^p)< \varepsilon.\]
  \end{proof}

If $B \subset  \Op$, we define new (larger) classes as follows~:

\begin{itemize}
\item $\Lambda_1(B)$ is the set of operators $(T,\mathrm{id})\colon \dom(T) \oplus_p L_p(\Omega,\mu) \to \ran(T) \oplus L_p(\Omega,\mu)$ for $T \in B$ and a measure space $(\Omega,\mu)$.
\item $\Lambda_2(B) = \{ U \circ T \circ V \mid U,V \textrm{ spatial isometries, }T \in B\}$.
\item $\Lambda_3(B)$ is the set of all $S \colon \dom(S) \subset L_p(\Omega_1,m_1) \to L_p(\Omega_2,m_2)$ such that there is $T \in B$ where $\dom(T) \subset L_p(\Omega_1,m_1) \oplus L_p(\Omega,m)$, $\ran(T) \subset L_p(\Omega_2,m_2) \oplus L_p(\Omega,m)$, $\dom(S)$ is the image of $\dom(T)$ by the first coordinate projection and $T(f \oplus g) = Sf \oplus g$ for every $f\oplus g \in \dom(T)$.
\item $\Lambda_4(B)$ is the set of all $S \colon \dom(S) \subset L_p(\Omega,m) \to L_p(\Omega',m')$ such that for every finite family $f_1,\dots,f_n$ in the domain of $T$ and every $\varepsilon>0$, there is $T \in B$ with domain contained in $L_p(\Omega,m)$ and range contained in $L_p(\Omega',m')$ and elements $g_1,\dots,g_n \in D(S)$ such that $\|f_i - g_i\|\leq \varepsilon$ and $\| T f_i - S g_i \|\leq \varepsilon$.
\end{itemize}
To save place, we denote $\Lambda_{123}(B) = \Lambda_3(\Lambda_2(\Lambda_1(B)))$.

\begin{cor}\label{cor:whenmu_f-muTf} For every $T \in \Op$ and $B \subset \Op$, the following are equivalent:
  \begin{itemize}
  \item for every $n$, $P(T,n) \subset P(B,n)$.
  \item The restriction of $T$ to every finite dimensional subspace of $\dom(T)$ belongs to $\Lambda_{123}(B)$.
  \end{itemize}
\end{cor}
\begin{proof}
  Assume that, for a fixed $n$, $P(T,n) \subset P(B,n)$. This means that, for every $f\in \dom(T)^n$, there is $S \in B$ and $g \in \dom(S)^n$ such that $\mu_f - \mu_{Tf} = \mu_g - \mu_{S_g}$. By Lemma \ref{lem:when_mu-mu=mu-mu} and Lemma \ref{lem:when_muf=mug}, there are $h \in L_p(\Omega,m)^n$ and $\overline{h}\in L_p(\Omega',m')^n$ and spatial isometries $U \colon \vect\{ S g_i \oplus \overline{h}_i\} \to \vect\{ T f_i \oplus h_i\}$ sending $S g_i \oplus \overline{h}_i$ to $T f_i \oplus h_i$ and $V \colon \vect\{ f_i \oplus h_i\to g_i \oplus \overline{h}_i\}$ sending $f_i \oplus h_i$ to $g_i \oplus \overline{h}_i$. The operator $S_1=(S,\mathrm{id})$ on $\dom(T) \oplus L_p(\Omega',m')$ belongs to $\Lambda_1(B)$, so the operator $S_2 = U \circ S_1 \circ V$, which sends $f_i \oplus h_i$ to $Tf_i \oplus h_i$ belongs to $\Lambda_2(\Lambda_1(B))$, and therefore the restriction of $T$ to $\vect\{f_1,\dots,f_n\}$ belongs to $\Lambda_3(\Lambda_2(\Lambda_1(B)))$. This proves one direction.

  The converse is simpler: it follows from the easy directions in Lemma \ref{lem:when_mu-mu=mu-mu} and Lemma \ref{lem:when_muf=mug} that $P(\Lambda_i(B),n) = P(B,n)$ for $i=1,2,3$. In particular, if the restriction of $T$ to every $\leq n$-dimensional subspace of $\dom(T)$ belongs to $\Lambda_{123}(B)$, then $P(T,n) \subset P(B,n)$.
\end{proof}

\subsection{Convergences in $H_n^*$} This section is devoted to the understanding of the encoding of both  weak-* sequential convergence and norm convergence in $H_n^*$. Our first result asserts that weak-* convergence of sequences corresponds to the operation $\Lambda_4$ we just defined.

\begin{prop}\label{prop:sequential_weak*-closure} Let $B \subset \Op$. The smallest class containing $B$ and stable by all operations $\Lambda_1,\Lambda_2,\Lambda_3,\Lambda_4$ coincides with the set of $T \in \Op$ such that for every $n$, $P(T,n)$ is contained in the sequential weak-* closure of $P(B,n)$.
\end{prop}
\begin{proof}
  We define by transfinite induction, for every ordinal $\alpha$, a class $B_\alpha$ as follows. $B_0$ is $\Lambda_{123}(B)$. If $\alpha$ is a successor ordinal, $B_{\alpha} =\Lambda_{123}(\Lambda_4(B_{\alpha-1}))$. If $\alpha$ is a limit ordinal we set $B_\alpha = \cup_{\beta<\alpha} B_\beta$.

  Similarly, we define, for every integer $n$ and every ordinal $\alpha$, a subset $C^n_\alpha \subset H_n^*$ by $C^n_0 = P(B,n)$, for a successor ordinal $C^n_{\alpha}$ is the set of all limits of weak-* converging sequences of elements of $C^n_{\alpha-1}$. If $\alpha$ is a limit ordinal we set $C^n_\alpha = \cup_{\beta<\alpha} C^n_\beta$.

  We claim that, for every $T \in \Op$ with $\dom(T)$ finite-dimensional, $P(T,n) \subset C^n_\alpha$ for every $n$ if and only if $T$ belongs to $B_\alpha$. We prove it by transfinite induction. If $\alpha=0$, this is Corollary~\ref{cor:whenmu_f-muTf}. Let $\alpha>0$ and assume that the claim holds for all $\beta<\alpha$. If $\alpha$ is a limit ordinal, the claim is clear.

  So assume that $\alpha$ is a successor. Assume first that $P(T,n) \subset C^n_{\alpha}$ for every $n$. Let $n$ be the dimension of $\dom(T)$ and $f=(f_1,\dots,f_n)$ a basis. Then $\mu_f - \mu_{Tf}$ is a limit of a weak-* converging sequence $\nu_k$ of elements of $C^n_{\alpha-1}$. By Lemma \ref{lem:dual_Hn} there are $f^{(k)} \in L_p(\Omega_k,m_k)^n$ and $g^{(k)} \in L_p(\Omega'_k,m'_k)^n$ with values in $\{z \in \K^n, |z|=1\}$ such that $\nu_k = \mu_{f^{(k)}} - \mu_{g^{(k)}}$ and $m_k(\Omega_k)+m'_k(\Omega'_k)$ is the norm of the corresponding linear form, which is bounded by Banach-Steinhaus. For simplicity of the exposition assume that $m_k(\Omega_k)+m'_k(\Omega'_k) \leq 1$. By the induction hypothesis, there is an operator $S_k \in B_{\alpha-1}$ such that $f^{(k)} \in D(S_k)^n$ and $S_k  f^{(k)} =g^{(k)}$. We have two sequences of probability measures, $ f^{(k)}_* m_k + (1-m_k(\Omega_k))\delta_0$ and $ g^{(k)}_* m'_k + (1-m'_k(\Omega'_k))\delta_0$, on $\{0\} \cup \{z \in \K, |z| =1\} \subset \K^n$. By compactness, up to an extraction we can assume that both sequences converge weak-*, and by Skorohod's representation theorem we can assume that $(\Omega_k,m_k)$ does not depend on $k$ and that $f^{(k)}$ converges almost surely to some $ f^{(\infty)} \in L_p^n$ and similarly $g^{(k)}$ converges almost surely, and in particular in $L_p$, to $g^{(\infty)}$ (this modifies the operators $S_k$, but they still satisfy $\nu_k = \mu_{f^{(k)}} - \mu_{S f^{(k)}}$ and therefore still belong to $B_{\alpha-1}$). In particular, the operator $S^{(\infty)}$ from $\dom(S^{(\infty)}) = \vect\{f^{(\infty)}_1,\dots,f^{(\infty)}_n\} \to \vect\{ g^{(\infty)}_1,\dots,g^{(\infty)}_n\}$ sending $f^{(\infty)}_i$ to $g^{(\infty)}_i$ belongs to $\Lambda_4(B_{\alpha-1})$ and it satisfies
  \[ \mu_{f^{(\infty)}} - \mu_{S^{(\infty)}f^{(\infty)}} = \lim_k \mu_{f^{(k)}} - \mu_{S^{(k)}f^{(k)}} = \mu_f - \mu_{Tf}.\]
  By Corollary \ref{cor:whenmu_f-muTf} again, the restriction of $T$ to $\vect\{f_1,\dots,f_n\} = E$ belongs to $\Lambda_{123}(S^{(\infty)}) \subset B_\alpha$.
This concludes the proof that $P(T,n) \subset C^n_{\alpha}$ for all $n$ implies that the restriction of $T$ belongs to $B_\alpha$. The converse is similar but easier and left to the reader.
\end{proof}

Similarly, norm convergence is well encoded.
\begin{lem}\label{lem:norm_convergence} Let $B \subset \Op$, $T \colon \dom(T) \subset L_p(\Omega_1,m_1) \to L_p(\Omega_2,m_2)$ a linear map with domain of finite dimension $n$, and $(f_1,\dots,f_n)$ be a basis of $\dom(T)$. Then $P(T,n)$ is contained in the norm-closure of $P(B,n)$ if and only if for every $\varepsilon>0$, there is $S \in \Lambda_{123}(B)$ with $\dom(S) \subset L_p(\Omega_1 \times [0,1],m_1 \otimes d\lambda)$ and $\ran(S) \subset L_p(\Omega_2 \times [0,1],m_2 \otimes d\lambda)$, there are $g_1,\dots,g_n \in \dom(S)$ such that
  \[ \int_{\Omega_1\times[0,1]} (|f(\omega)|^p+|g(\omega,s)|^p) \chi_{f(\omega) \neq g(\omega,s)} dm_1(\omega) ds \leq \varepsilon\]
  and
  \[ \int_{\Omega_2 \times [0,1]} (|Tf(\omega)|^p + |Sg(\omega,s)|^p)^{\frac 1 p} \chi_{Tf(\omega) \neq Sg(\omega,s)} dm_2(\omega) ds \leq \varepsilon.\]
\end{lem}
\begin{proof}
  Assume that $P(T,n)$ is contained in the norm closure of $P(B,n)$. This means that for every $\varepsilon>0$, there $\mu' \in P(B,n)$ such that $\|\mu_f - \mu_{Tf} - \mu'\| < \varepsilon$. By Lemma~\ref{lem:jordan_decomposition}, we can write $\mu'=\mu_g-\mu_h$ for $n$-uples of elements of $L^p$ spaces $g,h$ where $\|\mu_f - \mu_g\|+\|\mu_{Tf} - \mu_h\| < \varepsilon$. Since we have some room $(<\varepsilon$), we can even assume that $\{g_1,\dots,g_n\}$ are linearly independant, so that we can define a linear map $S$ sending $g_i$ to $h_i$. By Corollary~\ref{cor:whenmu_f-muTf}, $S$ belongs to $\Lambda_{123}(B)$, and so does $S$ composed with any spatial isometry. So the only if direction follows from Lemma~\ref{lem:when_muf_and_mug_are_close} and its improvement in Remark~\ref{rem:when_muf_and_mug_are_close} .

The converse is proved the same way.
\end{proof}

\subsection{Proof of the main Theorem}
We are also ready to prove our main Theorem~\ref{thm:bipolar_of_operator_explicit}. Before we do so, we only need to understand the operation of taking convex hulls.
\begin{lem}\label{lem:understandP} Let $B\subset \Op$ and $n \in \N$. The convex hull of $P(B,n)$ is equal to $P(\oplus_{\ell_p}(B),n)$ where $\oplus_{\ell_p}(B)$ is the class of all finite $\ell_p$-direct sums of operators in $B$.
\end{lem}
\begin{proof}
This is clear: if $T_1,\dots,T_k \in B$ and $f^{(j)} \in D(T_j)^n$ for all $j$, then 
\[ \sum_j \mu_{f^{(j)}} - \mu_{Tf^{(j)}} = \mu_{f} - \mu_{(T_1\oplus \dots \oplus T_k)f}\]
where $f_i = f_i^{(1)}\oplus \dots \oplus f_i^{(k)} \in D(T_1)\oplus \dots D(T_k)$ and $f = (f_1,\dots,f_n) \in (D(T_1)\oplus \dots D(T_k))^n$.
\end{proof}

We can conclude.
\begin{proof}[Proof of Theorem~\ref{thm:bipolar_of_operator_explicit}]
  We start with the easy direction. Assume that for every $n$ and $\varepsilon$, the assumption in the second bullet point holds. Let $X$ be Banach space such that $\sup_{S \in B} \|S_X\|\leq 1$. We have to prove that $\|T_X\|\leq 1$. That is, for every integer $n$ and every $x_1,\dots,x_n$,
  \begin{equation}\label{eq:TX_norm1} \|\sum_i (Tf_i) x_i\|_{L_p(\Omega_2;X)} \leq \|\sum_i f_i x_i\|_{L_p(\Omega_1,X)}.\end{equation}
  Let $\varepsilon>0$, and $S=S_0\oplus S_1 \dots S_k$, $U$, $V$, $g_i,g'_i,h_i$ given by the assumption. In the following computation, we view $Tf_i$ as an element of $L_p(\Omega_2\times[0,1])$ that does not depend on the second variable in $[0,1]$, and similarly for $f_i$. We denote simply by $\|\cdot\|_p$ the norm in $L_p(\Omega_i\times[0,1];X)$ or $L_p(\Omega_i\times[0,1])$. We can bound
  \begin{align*}
    \|\sum_i (Tf_i) x_i\|_{L_p(\Omega_2;X)} &\leq \sum_i \|Tf_i - g'_i\|_p \|x_i\|+\|\sum_i g'_i x_i\|_{p}\\
    & \leq \varepsilon \sum_i \|x_i\| + \|\sum_i g'_i x_i\|_{p}\\
    & = \varepsilon \sum_i \|x_i\| + \left(\|\sum_i (g'_i,h_i) x_i\|_{p}^p - \|\sum_i h_i x_i\|_p^p\right)^{\frac 1 p}.
  \end{align*}
  The quantity inside the parenthesis is equal to
  \[ \|\sum_i S (g_i,h_i) x_i\|_p^p - \|\sum_i h_i x_i\|_p^p,\]
  so using that $\|S_X\| = \max_{0 \leq i \leq k} \|(S_i)_X\|\leq 1$, we obtain that it is bounded above by
    \[ \|\sum_i (g_i,h_i) x_i\|_p^p - \|\sum_i h_i x_i\|_p^p= \|\sum_i g_i x_i\|_p^p.\]
    We can therefore go on with our computation and get
      \begin{align*}
        \|\sum_i (Tf_i) x_i\|_{L_p(\Omega_2;X)} & \leq \varepsilon \sum_i \|x_i\| + \|\sum_i g_i x_i\|_p\\
        &  \leq \varepsilon \sum_i \|x_i\| + \sum_i \|f_i - g_i\|_p \|x_i\|+ \|\sum_i f_i x_i\|_p\\
        & \leq 2 \varepsilon \sum_i \|x_i\|+ \|\sum_i f_i x_i\|_p.
      \end{align*}
      Making $\varepsilon\to 0$, we obtain \eqref{eq:TX_norm1} as required.

      The converse direction relies on everything we have obtained so far. Assume that $T \in (\prep{B})^\circ$. We know from Corollary~\ref{cor:bipolarsbis} that for every integer $n$, $P(T,n) \subset \overline{\mathrm{conv}}^{\|\cdot\|}P(B\cup REG,n)$, which is the same as the norm-closure of $P(\oplus_{\ell_p}(B \cup REG))$ by Lemma~\ref{lem:understandP}. So by Lemma~\ref{lem:norm_convergence}, for every $\varepsilon>0$ there is $S \in \Lambda_{123}(B \cup REG)$ with $\dom(S) \subset L_p(\Omega_1 \times [0,1],m_1 \otimes d\lambda)$ and $\ran(S) \subset L_p(\Omega_2 \times [0,1],m_2 \otimes d\lambda)$, there are $g_1,\dots,g_n \in \dom(S)$ such that
  \[ \int_{\Omega_1\times[0,1]} (|f(\omega)|^p+|g(\omega,s)|^p) \chi_{f(\omega) \neq g(\omega,s)} dm_1(\omega) ds \leq \varepsilon^p\]
  and
  \[ \int_{\Omega_2 \times [0,1]} (|Tf(\omega)|^p + |Sg(\omega,s)|^p)^{\frac 1 p} \chi_{Tf(\omega) \neq Sg(\omega,s)} dm_2(\omega) ds \leq \varepsilon^p.\]
In particular, using that $\int (|a|^p+|b|^p) \chi_{a \neq b} \geq \int |a-b|^p = \sum_i \int |a_i - b_i|^p$ for every $a,b \in (L_p)^n$, we have for every $i$,
  \[ \left(\int_{\Omega_1\times[0,1]} |f_i(\omega)-g_i(\omega,s)|^p dm_1(\omega) ds\right)^{\frac 1 p} \leq \varepsilon\]
  and
  \[ \left(\int_{\Omega_2 \times [0,1]} |Tf_i(\omega) - Sg_i(\omega,s)|^p dm_2(\omega) ds\right)^{\frac 1 p}  \leq \varepsilon.\]
  Also, using that $REG$ contains the identity and is stable by $\ell_p$-direct sums, $\Lambda_1(\oplus_{\ell_p} (B \cup REG))$ is the set of all operators of the form $S_0 \oplus S_1 \oplus \dots \oplus S_k$ for $S_0 \in REG$ and $S_1,\dots,S_k \in B$. So the fact that $S$ belongs to $\Lambda_{123}(B \cup REG)$ means that there exist $S_0 \in REG$, $S_1,\dots,S_k \in B$, spatial isometries $U,V$ and a measure space $(\Omega,m)$ such that $V \circ (S_0\oplus\dots \oplus S_k) \circ V$ contains elements of the form $(g_i,h_i)$ in its domain for all $1 \leq i \leq n$ and some $h_i \in L_p(\Omega,m)$ and so that $V \circ (S_0\oplus\dots \oplus S_k) \circ V (g_i,h_i) = (S g_i, h_i)$. We can of course replace $(\Omega,m)$ by any standard measure space as this amounts to conjugating by another spatial isometry. So we have obtained the conclusion of the theorem.
\end{proof}

\section{Variants of the duality}\label{section:comparison}

The duality defined in the Introduction was implicit in many early works on the geometry of Banach spaces, and was essentially present in \cite{MR2732331}, where Pisier explicitly considered a duality that is very close to ours. He defines the polars by the same formulas as in Definition~\ref{defn:polar} and \ref{defn:prepolar}, but with different classes of operators instead of $\Op$.
\begin{defn} Denote by $\Opf$ the class of all linear operators between (full, as opposed to subspace of) $L_p$ spaces. If $n$ is an integer, denote $\Opfn$ the class of all linear operators $\ell_p^n \to L_p$ and $\Opff = \cup_{n \in \N} \Opfn$.
\end{defn}
In particular, for the duality between $\Banach$ and $\Opf$, the polar of a set $B$ of Banach spaces is smaller than for the duality between $\Banach$ and $\Op$, and therefore its bipolar is larger. Hernandez also obtained a description of the bipolar $B$ for this duality: it is the set of Banach spaces that are finitely representable in subspaces of quotients of finite $\ell_p$ direct sums of spaces in $B$, see Theorem~\ref{thm:hernandezbis}. This is quite different from Theorem \ref{thm:hernandez}. For example for the duality considered here, every Banach space in the bipolar of $\ell_1$ has cotype $\max(p,2)$ (this is immediate from Hernandez's Theorem \ref{thm:hernandez} and the fact that $\ell_p(\ell_1)$ has cotype $\max(p,2)$), whereas for the duality in \cite{MR2732331}, the bipolar of $\ell_1$ contains every space finitely representable in a quotient of $\ell_1$, \emph{i.e.} every Banach space.

However, as far as the bipolar of a set of operators is concerned, the two dualities are very related~: if $B \subset \Opf$, then its bipolar for the polarity between $\Banach$ and $\Opf$ is the set of operators between $L_p$ spaces which belong to $\prep B^{\circ}$ (for the polarity between $\Banach$ and $\Op$). So our Theorem~\ref{thm:bipolar_of_operator_explicit} also provides an answer to \cite[Problem 4.1]{MR2732331}.

The dualities discussed so far are isometric variants of two other isomorphic forms of the duality in \cite{MR2732331}, where $A^\circ$ is the class of operators such that $\|T_X\|<\infty$ for all $X \in A$, and $\prep{B}$ is the class of Banach space such such $\|T_X\|<\infty$ for all $T \in B$. But, if $B$ is finite, the bipolar of $B$ for this ``isomorphic'' duality coincides with $\cup_{R>0} R\prep{(R^{-1} B)}^\circ$. If $B$ is infinite, the ``isomorphic'' bipolar of $B$ is $\cup_{R>0} \cup_{B'} R (\prep{B'})^\circ$, where $B' = \{ \{c_T T \mid T \in B\} \mid c \in (0,1]^B\}$. So our bipolar Theorem~\ref{thm:bipolar_of_operator_explicit} also allows to describe the bipolar for the isomorphic forms of the duality.

It turns out that the methods of this paper also allow us to recover Hernandez's characterization of the bipolar of a set of Banach spaces for the duality between $\Banach$ and $\Opf$. Let us start with an easy fact, which allows us to resctrict our attention to $\Opfn$.
\begin{lem}\label{lem:finite_dim_restriction} The bipolar of a subset $A \subset \Banach$ for the duality between $\Banach$ and $\Opf$ coincides with its bipolar for the  duality between $\Banach$ and $\Opff$. 
\end{lem}
\begin{proof} Any $L_p$ space can be written as the closure of an increasing net of finite dimensional $L_p$ spaces, which are isomorphic to $\ell_p^n$.
\end{proof}
Denote by $e=(e_1,\dots,e_n)$ the standard basis of $\ell_p^n$. We encode a subset $B \subset \Opfn$, as the cone $\widetilde{P}(B,n) \subset H_n^*$
\[ \widetilde{P}(B,n) = \{s(\mu_e - \mu_{Te}) \mid T \in B,s >0\}.\]
We warn the reader that $\widetilde{P}(B,n)$ is strictly smaller than $P(B,n)$ even for $B \subset \Opfn$. Note however that the duality is still efficiently encoded. The following result is the analogue of Lemma~\ref{lem:polarity_basic} and is proved identically.
\begin{lem}\label{lem:polarity_basic_bis} Let $n \in \N$ be an integer, $A \subset \Banach$ be a class of Banach spaces and $B \subset \Opfn$ a class of operators $\ell_p^n \to L_p$. 
\begin{enumerate}
\item\label{item:Pofpolarbasicbis}$B \subset A^\circ$ if and only if $\widetilde{P}(B,n) \subset N(A,n)^\circ$. 
\item\label{item:Nofpolarbasicbis} $A \subset \prep{B}$ if and only if $N(A,n) \subset \prep{\widetilde{P}(B,n)}$.
\end{enumerate}
\end{lem}
Moreover, we have
\begin{lem}\label{lem:polar_of_Ptilde} The subset $\widetilde{P}(\Opfn,n) \subset H_n^*$ is a weak-* closed convex cone. Its polar is
  \[ C_n:=\{ \varphi \in H_n \mid \varphi \leq 0, \varphi(e_1)=\dots=\varphi(e_n)=0\}.\]
\end{lem}
\begin{proof}The convexity of the cone $\widetilde{P}(\Opfn,n)$ is clear, as $\theta (\mu_e - \mu_{Te}) + (1-\theta)(\mu_e - \mu_{Se})$ can be written as $\mu_e - \mu_{Re}$ for the linear map $R \colon \ell_p^n \to L_p \oplus L_p$ given in matrix form by $R = \begin{pmatrix} \theta^{\frac 1 p} T \\ (1-\theta)^{\frac 1 p} s \end{pmatrix}$. For the weak-* closedness, by the Krein-Smulian theorem \cite[Theorem V.12.1]{MR1070713}, we have to show that the intersection of $\widetilde{P}(\Opfn,n)$ with the closed unit ball $B_{H_n^*}$ is sequentially weak-* closed. Recall that every element of $H_n^*$ can be regarded as a signed measure on $\KP^{n-1}$. If it belongs to $\widetilde{P}(\Opfn,n) \cap B_{H_n^*}$, then its positive part in the Jordan decomposition has total mass $\leq 1$ and has support contained in $\{\K e_1,\dots, \K e_n\}$. In particular, it is less than $\mu_e$. It follows that it can be written as $\mu_e - \mu_{f}$ for some $f \in (L_p)^n$. So we are reduced to showing that $\{\mu_e - \mu_f \mid f \in L_p^n\}$ is weak-* closed in $H_n^*$, which is clear.

  It remains to identify the polar of $\widetilde{P}(\Opfn,n)$.
  If $\varphi \in C_n$, and $T \in \Opfn$, we have
  \[ \langle \mu_e ,\varphi\rangle = \sum_i \varphi(e_i) = 0,\]
  so 
  \[ \langle \mu_e - \mu_{Te},\varphi\rangle = -\langle  \mu_{Te},\varphi\rangle \geq 0.\]
  This shows the inclusion $C_n \subset \prep{\widetilde{P}(\Opfn,n)}$.

  For the converse inclusion, consider $\varphi \in \prep{\widetilde{P}(\Opfn,n)}$. Then for every $f \in (L_p)^n$, we have
  \begin{equation}\label{eq:polar_of_Ptilde} \langle \mu_e - \mu_f,\varphi\rangle \geq 0.
  \end{equation}
  In particular, replacing $f$ by $sf$ and making $s\to \infty$, we obtain 
  \[ -\langle \mu_f,\varphi\rangle \geq 0\]
  for every $f \in (L_p)^n$. Taking for $f$ a constant $z$, this forces $\varphi(z) \leq 0$ for every $z \in \K^n$. Taking $f=0$ in \eqref{eq:polar_of_Ptilde} leads to
  \[ \sum_i \varphi(e_i) = \langle \mu_e,\varphi\rangle \geq 0.\]
  This implies that $\varphi(e_i) = 0$ for every $i$, and that $\varphi$ belongs to $C_n$. This concludes the proof of the inclusion $\prep{\widetilde{P}(\Opfn,n)} \subset C_n$ and of the lemma.
\end{proof}
We can now reprove Hernandez' Theorem.
\begin{thm}\label{thm:hernandezbis}(\cite{MR703903}) Let $A \subset \Banach$ and $X \in \Banach$. The following are equivalent.

  \begin{enumerate}
  \item\label{item:hernandezbispolar} For every operator $T \colon L_p \to L_p$, $\polar{T}{X} \leq \sup_{Y \in A} \polar{T}{Y}$.
    \item\label{item:hernandezbisfinrep} $X$ is finitely representable in the class of all quotients of finite $\ell_p$-direct sums of elements in $A$.
  \end{enumerate}

\end{thm}
  \begin{proof} The interesting direction is (\ref{item:hernandezbispolar})$\implies$(\ref{item:hernandezbisfinrep}). So assume that (\ref{item:hernandezbispolar}) holds. Equivalently by Lemma~\ref{lem:finite_dim_restriction} $X$ belongs to the bipolar of $A$ for the duality between $\Banach$ and $\Opff$. By Lemma~\ref{lem:polarity_basic_bis}, this holds if and only if for every $n$,
    \[ N(X,n) \subset \prep{(\widetilde{P}(\Opfn,n) \cap N(Y,n)^\circ)}.\]
    By Lemma~\ref{lem:polar_of_Ptilde} and the bipolar theorem, $\widetilde{P}(\Opfn,n)$ coincides with $C_n^\circ$, so the previous inclusion becomes
    \[ N(X,n) \subset \prep{((C_n \cup N(Y,n))^\circ)}.\]
    By the bipolar theorem again, we obtain that $N(X,n)$ belongs to the closed convex hull of $ C_n \cup N(Y,n)$, which is nothing but $C_n + \overline{\mathrm{conv}} (N(Y,n))$ (use compactness as in Lemma~\ref{lem:compactnessNAn} to see that $C_n + \overline{\mathrm{conv}} (N(Y,n))$ is closed). By Lemma~\ref{lem:understandN}, we obtain that for every $n \in \N$ and every $x_1,\dots,x_n \in X$, there is a space $Y$ finitely representable in the finite $\ell_p$-direct sums of elements in $A$ and elements $y_1,\dots,y_n$ spanning $Y$ such that
    \[ \forall i, \|x_i\| = \|y_i\| \textrm{ and }\forall z \in \K^n,\|\sum_i z_i x_i\|_X \leq \| \sum_i z_i y_i\|_Y.\]
    Now if $E$ is any finite dimensional subspace of $X$, and $\varepsilon>0$, we can pick a finite family $x_1,\dots,x_n$ in its unit sphere whose convex hull contains the ball of radius $(1+\varepsilon)^{-1}$. Applying the preceding to these $x_i$'s, we obtain  a space $Y$ finitely representable in $\oplus_{\ell_p} A$ and a linear map $u \colon Y \to E$ of norm $1$ such that the image of the unit ball contains the ball of radius $(1+\varepsilon)^{-1}$ of $E$. In other words, $E$ is at Banach-Mazur distance $\leq 1+\varepsilon$ from a quotient of $Y$. But a subspace of a quotient is the same as a quotient of a subspace, so we have obtained that $X$ is finitely representable in the quotients of spaces in $\oplus_{\ell_p} A$. This is (\ref{item:hernandezbisfinrep}).

The converse implication (\ref{item:hernandezbisfinrep})$\implies$(\ref{item:hernandezbispolar}) can be proved using the same arguments, but it is easy and classical to check it directly. The point that perhaps deserves a small justification is why $\polar{T}{X} \leq \polar{T}{Y}$ if $X$ is a quotient of $Y$ and $T \colon L_p \to L_p$ is an operator. One argument is by duality. Indeed, $X^*$ identifies then as a subspace of $Y^*$, and if $T^* \colon L_q \to L_q$ denotes the dual of $T$ (for $\frac 1 q + \frac 1 p=1$), then
    \[\polar{T}{X} = \polar{T^*}{X^*} \leq \polar{T^*}{Y^*} = \polar{T}{Y}.\]
    \end{proof}
\appendix
\section{On the $\GL(n,\K)$ invariant subspaces of the space of homogeneous functions}

Let $0<p<\infty$. We recall some definition that already appeared in the body of the paper for $p \geq 1$.

Let $n$ be a positive integer. Denote by $|z|$ the $\ell_p$-''norm'' on $\K^n$
\[ |z| = \left( |z_1|^p+ \dots + |z_n|^p \right)^{\frac 1 p}.\]

A function $\varphi \colon \K^{n} \to \R$ is called homogeneous of degree $p$ if $\varphi(\lambda z) = |\lambda|^p \varphi(z)$ for all $z \in \K^n$ and $\lambda \in \K$. The space $H_n$ of real continuous homogeneous of degree $p$ functions on $\K^n$ is a Banach space for the topology of uniform convergence on compact subsets on $\K^n$. A particular choice of norm is $\|\varphi\|= \sup_{|z| \leq 1} |\varphi(z)|$, so that for this norm $H_n$ is isometrically isomorphic to the space of continuous functions on $\KP^{n-1}$ through the identification of $\varphi \in H_n$ with the function $\K z \in \KP^{n-1} \mapsto \varphi\left(\frac{z}{|z|}\right)$. For this identification, the natural action of $\GL_n(\K)$ on $H_n$ corresponds to the action of $\GL_n(\K)$ on $C(\KP^{n-1})$ given by
\[ A \cdot \varphi(\K z) = \frac{|A^{-1}z|^p}{|z|^p}\varphi(A^{-1} \K z).\]

\begin{thm}\label{thm:Gln_invariant_subspaces} The $\GL_n(\K)$-invariant closed subspaces of the Banach space $H_n$ of continuous $p$-homogeneous functions $\K^n \to \R$ are 
\begin{itemize}
\item $\{0\}$ and $H_n$ if $p$ is not an even integer.
\item $\{0\}$, $H_n$ and the subspace of degree $p$ homogeneous polynomials if $p$ is an even integer.
\end{itemize}
\end{thm}

\begin{rem}\label{rem:spatial_and_regular_isometries}
This theorem allows to reprove the result \cite{MR611233} that if $p$ is not an even integer, then every isometry between subspaces of $L_p$ spaces is a spatial isometry. Indeed, if $T$ is such an isometry, $n$ is an integer, $f \in D(T)^n$  and $\varphi (z) = |z_1|^p$, then we get for every $A\in \GL_n(\K)$  (with the notation of \eqref{eq:def_muf})
\[\dual{\mu_f-\mu_{Tf}}{\varphi \circ A} = \|\sum_j a_{1,j} f_j\|^p -  \|\sum_j a_{1,j} Tf_j\|^p =0.\]
The linear form $\mu_f - \mu_{Tf}$ therefore vanishes on the $\GL_n(\K)$-invariant subspace spanned by $\{\varphi \circ A,A \in \GL_n(\K)\}$. By Theorem \ref{thm:Gln_invariant_subspaces} this subspace is dense, which implies that $\mu_f-\mu_{Tf}=0$. One concludes by Remark \ref{rem:when_muf=mug} that $T$ is a spatial isometry.

When $p$ is an even integer, the same argument shows that if $X$ is a Banach space and $x,y \in X$ are so that $(z_1,z_2) \mapsto \|z_1 x+z_2 y\|^p$ is not a polynomial in $z_1,z_2,\overline z_1,\overline z_2$ (for example if $X = \K^2$ with the $\ell_q$ norm for $q$ which is not an even divisor of $p$), then every operator $T$ between subspaces of $L_p$ spaces such that $\|T_X\|=\|T^{-1}_X\|=1$ is a spatial isometry. In particular we have:\end{rem}
\begin{cor}\label{cor:spatial_and_regular_isometries} For any $0<p<\infty$ (even integer or not) a linear map $T$ between subspaces of $L_p$ spaces is a spatial isometry if and only if $T$ is a regular isometry.
\end{cor}
Rudin's proof in \cite{MR0410355} relied on the Wiener Tauberian theorem. In the proof of Theorem \ref{thm:Gln_invariant_subspaces}, we shall need the following variant. 
\begin{prop}\label{prop:wiener} Let $f,g \colon \R^d \to \C$ be two measurable functions and $C>0$ such that $|f(x)| \leq C (1+|x|)^p$ and $|g(x)| \leq C (1+|x|)^{-p-d-1})$ for all $x \in \R^d$. Assume that $g \ast f = 0$. Then the support of the tempered distribution $\hat f$ is contained in $\{\xi \in \R^d,\hat g(\xi) = 0\}$.
\end{prop}
\begin{proof} First observe that the assumption on $g$ implies that $g \in L_1(\R^d)$.
  
If $g$ belongs to $\cD(\R^d)$ (the space of compactly supported $C^\infty$ functions), then the proposition is easy~: by taking Fourier transform we have $\hat g \hat f=0$ (multiplication of a distribution by a $C^\infty$ function), from which the conclusion follows. The strategy will be to approximate $g$ by compactly supported $C^\infty$ functions.

We have to prove that for every $\xi \in \R^d$ with $\hat g(\xi) \neq 0$, there is a neighbourhood $V$ of $\xi$ such that $\dual{\hat f}{\varphi} =0$ for every $\varphi \in \cD(V)$. By standard translation/convolution/dilation arguments, we can assume that $\xi=0$, $g$ is $C^\infty$, and that $\hat g$ does not vanish on the closure of $B(0,1)$. We will prove that $\dual{\hat f}{\varphi} =0$ for every $\varphi \in \cD(B(0,1))$.

Let $\rho \colon \R^d \to [0,1]$ be a compactly supported $C^\infty$ function, equal to $1$ on $B(0,1)$, and define a sequence of functions $g_n \in \cD(\R^d)$ by $g_n(x) = g(x) \rho(\frac x n)$. By the dominated convergence theorem, $\|g_n - g\|_{L_1(\R^d)} \to 0$, and so $\|\hat g_n - \hat g\|_{L_\infty} \to 0$. In particular there exists $n_0$ such that  $\hat g_n$ does not vanish on $B(0,1)$ for all $n \geq n_0$.

Let $\varphi \in \cD(B(0,1))$. Then $\frac{\varphi}{\hat g_n}$ belongs to $\cD(B(0,1))$, so we can write
\[ \dual{\hat f}{\varphi} = \dual{\hat g_n \hat f}{\frac{\varphi}{\hat g_n}} = \dual{g_n \ast f}{\cF^{-1}(\frac{\varphi}{\hat g_n})}\]
where $\cF^{-1}$ is the inverse Fourier transform. Using that $g_n \ast f(x) = (g_n-g)\ast f (x) =O(\frac 1 n (1+ \frac{|x|}{n})^p)$ (this inequality will be explained below), we get
\begin{equation}\label{eq:Ff_small} |\dual{\hat f}{\varphi}| \leq \frac{C}{n} \int (1+ \frac{|x|}{n})^p |\cF^{-1}(\frac{\varphi}{\hat g_n})| dx.\end{equation}
To justify to domination of $(g_n-g)\ast f (x) = \int (g_n-g)(y) f(x-y) dy$, use that \[|(g_n-g)(y)| \lesssim (1+|y|)^{-p-d-1} 1_{|y|>n}\] and \[|f(x-y)| \lesssim (1+|x-y|)^p \lesssim (1+\max(|x|,|y|))^p\] to obtain
\[ |f \ast (g-g_n)(x)| \lesssim \int_{|y|>n} (1+|y|)^{-p-d-1} (1+\max(|x|,|y|))^p dy.\]
If $|x| \leq n$, then the preceding inequality becomes
\[ |f \ast (g-g_n)(x)| \lesssim \int_{|y|>n} (1+|y|)^{-d-1} dy \lesssim \frac 1 n.\]
If $|x|\geq n$, then we cut the integral as $\int_{n <|y| \leq |x|} + \int_{|x| <|y|}$ and get
\begin{eqnarray*} |f \ast (g-g_n)(x)| &\lesssim& \int_{n <|y| \leq |x|} \frac{(1+|x|)^p}{(1+|y|)^{p+d+1}} dy +  \int_{|y|>|x|} (1+|y|)^{-d-1} dy\\
& \lesssim& |x|^p \frac{1}{n^{p+1}} + \frac{1}{|x|}   \lesssim \frac{|x|^p}{n^{p+1}}.\end{eqnarray*}
This proves the announced inequality.

In view of \eqref{eq:Ff_small}, we see that our goal is to prove good integrability properties on the function $\cF^{-1}(\frac{\varphi}{\hat g_n})$, \emph{i.e.} good regularity properties of its Fourier transform $\frac{\varphi}{\hat g_n}$. To achieve this, we denote by $A(\R^d)$ the Fourier algebra of $\R^d$, \emph{i.e.} the Banach space $\cF (L_1(\R^d))$ for the norm $\|h\|_{A(\R^d)} = \| \cF^{-1} h\|_{L_1(\R^d)}$. The inequality
\begin{equation}\label{eq:A_fourier_algebra} \| h_1 h_2\|_{A(\R^d)} \leq \| h_1\|_{A(\R^d)} \| h_2\|_{A(\R^d)}\end{equation}
is the reason for the term ``algebra'' and is clear from the usual properties of convolution and Fourier transform. We have the following lemmas.
\begin{lem}\label{lem:wiener_quantitative} For every $\varphi \in \cD(B(0,1))$, there is a constant $C=C(\varphi)$ such that $\frac{\varphi}{\hat g_n}$ belongs to $A(\R^d)$ with norm $\leq C$ for all $n \geq n_0$.
\end{lem}
\begin{lem}\label{lem:derivative_A(Rd)} There is a constant $C'$ such that $D^\alpha \hat g_n$ belongs to $A(\R^d)$ with norm $\leq C'$ for all $n \in \N$ and $\alpha \in \N^d$, $|\alpha|< p+1$.
\end{lem}

These two lemmas, together with the Leibniz derivation rule and the fact that $A(\R^d)$ is a Banach algebra \eqref{eq:A_fourier_algebra}, imply that, for every $\varphi \in \cD(B(0,1))$, there is a constant $C$ such that $D^\alpha\frac{\varphi}{\hat g_n}$ belongs to $A(\R^d)$ with norm less than $C$ for all $n \geq n_0$ and $\alpha\in \N^d$, $|\alpha|<p+1$. Therefore, for every such $n$ and $\alpha$ we have
\[ \int |x^\alpha \cF^{-1}(\frac{\varphi}{\hat g_n})| dx \leq C.\]
This implies that, if $k$ is the unique integer in the interval $[p,p+1)$, then for every $n \geq n_0$
\[ \int (1+|x|)^p \cF^{-1}(\frac{\varphi}{\hat g_n})| dx \leq \int (1+|x|)^k \cF^{-1}(\frac{\varphi}{\hat g_n})| dx \leq C'.\]
A fortiori, by \eqref{eq:Ff_small} we have
\[ |\dual{\hat f}{\varphi}| \leq \frac{C'}{n},\]
so making $n \to \infty$ we obtain $\dual{\hat f}{\varphi}=0$. This concludes the proof.
\end{proof}
We have to prove the two lemmas used above.
\begin{proof}[Proof of Lemma \ref{lem:wiener_quantitative}] Let $\rho \in D(B(0,1))$ which is equal to $1$ on the support of $\varphi$. The fact that $\frac{\rho}{\hat g}$ (and $\frac{\rho}{\hat g_n}$ for every $n \geq n_0$) belongs to $A(\R^d)$ is essentially the Wiener tauberian theorem. Indeed, the proof in \cite[Theorem 9.3]{MR1157815} shows that for every $x \in \C$ such that $\hat g(x) \neq 0$, there is $\varepsilon>0$ such that $\frac{\rho}{\hat g} \in A(\R^d)$ for every $\rho \in D(B(x,\varepsilon))$. The claimed result follows by a partition of unity argument. To obtain a bound on $\frac{\rho}{\hat g_n}$ independant from $n$, we write
\[ \frac{\varphi}{\hat g_n} = \frac{\varphi}{\hat g} \frac{1}{1-\frac{\rho}{\hat g}(\hat g - \hat g_n)}.\]
Since $\frac{\rho}{\hat g}$ belongs to $A(\R^d)$ and $\|\hat g - \hat g_n\|_{A(\R^d)} = \|g-g_n\|_{L_1(\R^d)} \to 0$, there is $n_1 \geq n_0$ such that $\frac{\rho}{\hat g}(\hat g - \hat g_n)$ has $A(\R^d)$-norm less than $\frac 1 2$ for all $n\geq n_1$. This implies that  for $n \geq n_1$
\[ \frac{\varphi}{\hat g_n} =  \sum_{k \geq 0} \frac{\varphi}{\hat g} \left( \frac{\rho}{\hat g}(\hat g - \hat g_n)\right)^{k}\]
belongs to $A(\R^d)$ with norm less than $2 \|\frac{\varphi}{\hat g}\|_{A(\R^d)}$. The lemma follows with \[C = \max(2 \|\frac{\varphi}{\hat g}\|_{A(\R^d)},\max_{n_0 \leq n <n_1} \|\frac{\varphi}{\hat g_n}\|_{A(\R^d)})).\]
\end{proof}
\begin{proof}[Proof of Lemma \ref{lem:derivative_A(Rd)}]
We have 
\[ \|D^\alpha \hat g_n\|_{A(\R^d)} = \| x^\alpha g_n\|_{L_1(\R^d)} \leq  \| x^\alpha g\|_{L_1(\R^d)}\]
because $g_n(x) = g(x) \rho(\frac x n )$ and $0 \leq \rho \leq 1$. The quantity $\| x^\alpha g\|_{L_1(\R^d)}$ is finite because $g(x)=O(|x|^{-p-d-1})$ and $|\alpha|<p+1$.
\end{proof}

We can now prove the main result on $\GL_n(\K)$-invariant subspaces of $H_n$.
\begin{proof}[Proof of Theorem \ref{thm:Gln_invariant_subspaces}]
For simplicity we write the proof for $\K=\C$. The real case is similar, see Remark \ref{rem:Wiener_real}. Let $f_0 \in H_n$ be a nonzero function such that the space spanned by the functions $f_0 \circ A$ for $A \in \GL_n(\C)$ is not dense in $H_n$. We will prove that $p$ is an even integer and that $f_0$ is a homogeneous polynomial. By the Hahn-Banach theorem, there is a nonzero linear form $\varphi$ on $H_n$ which vanishes on $f_0 \circ A$ for all $A$. By the Riesz representation theorem, there is a unique nonzero signed measure $\mu$ on $\CP^{n-1}$ such that $\varphi(f) = \int f(\frac{z}{|z|}) d\mu(\C z)$. We can assume that $\mu$ is absolutely continuous with respect to the Lebesgue measure (=the unique $\U(n)$-invariant probability measure) on $\CP^{n-1}$, with a $C^\infty$ Radon-Nykodym derivative. Indeed, if $\rho$ is a $C^\infty$ function on $\U(n)$, then the measure $\rho \ast \mu = \int (u_* \mu)  \rho(u)du$ is absolutely continuous with respect to the Lebesgue measure on $\CP^{n-1}$, has a $C^\infty$ density, and still satisfies $\int f_0 \circ A(\frac{z}{|z|}) d(h \ast \mu)(\C z)=0$ for every $A \in \GL_n(\C)$. Moreover if $\rho \geq 0$ has a support which is a small enough neighbourhoud of the identity, then $\rho \ast\mu \neq 0$.

So in particular, $\mu$ has a nonzero bounded Radon-Nykodym derivative $h$ with respect to the Lebesgue measure. By Lemma \ref{lem:lebesgue_measure} we can write
\[ \int_{\CP^{n-1}} F(\C z) d\mu(z) = \int_{\C^{n-1}} (Fh)(\C(1,z)) \frac{c}{(1+|z_1|^{2} + \dots |z_{n-1}|^2)^n} dz.\]
Taking $F(\C z) = (f_0 \circ A)(\frac{z}{|z|})$, we get $F(\C(1,z)) = \frac{f_0 \circ A(1,z)}{1+|z|^p}$ and
\begin{equation}\label{eq:g} 0 = \int_{\C^{n-1}} f_0 \circ A(1,z) g(z) dz\end{equation}
for the nonzero function $g(z) = \frac{1}{1+|z|^p}\frac{d\mu}{d\lambda}(\C (1,z)) \frac{c}{(1+|z|_2^{2})^n}$, which satisfies.
\begin{equation}\label{eq:decay_of_g} g(z) = O((1+|z|)^{-p-2n}).\end{equation}

Now if we take for $A = \begin{pmatrix} 1&0\\ b & -A'^{-1} \end{pmatrix}$ for $A' \in \GL_{n-1}(\C)$, then \eqref{eq:g} becomes
\[ 0=\int f_0(1,b-A'^{-1} z) g(z) dz = |det A'|\int (g\circ A')(z) f_0(1,b-z) dz.\]
The second equality is a change of variable. In other words, if $f \colon \C^{n-1} \to \R$ is the function $f(z) = f_0(1,z)$, then $f$ is a continuous function satisfying $f(z) = O(1+|z|^p)$ as $z \to \infty$, and such that $(g\circ A') \ast f=0$ for every $A' \in \GL_{n-1}(\C)$.  

Viewing $\C^{n-1}$ as a real vector space $\R^{d}$ with $d=2n-2$, we see that we are in the setting of Proposition \ref{prop:wiener} ( \eqref{eq:decay_of_g} indeed implies that $(g\circ A')(z) = O((1+|z|)^{-p-d-2})$). So the proposition implies that the support of $\hat f$ is contained in $\{ \xi \in \C^{n-1}, \cF(g \circ A')(\xi) \neq 0\}$. But, $g$ being nonzero, there exists $\xi \neq 0$ such that $\hat g(\xi) \neq 0$. Since $\GL_{n-1}(\C)$ acts transitively on $\C^{n-1} \setminus \{0\}$, we get that the support of $\hat f$ is contained in $\{0\}$. This implies that $f$ is a polynomial function in $z,\overline z$. So we have proved that \eqref{eq:g} implies that the function $z \mapsto f_0(1,z)$ is a polynomial function in $z,\overline z$. But since \eqref{eq:g} for $f_0$ clearly implies  \eqref{eq:g} for $f_0 \circ A$ for every $A \in \GL_n(\C)$, we get that $z \mapsto f_0 \circ A(1,z)$ is a polynomial for every $A$. This implies that $p$ is an even integer and that $f_0$ is a homogeneous polynomial, see Lemma \ref{lem:p_even_integer}.

This shows that if $p$ is not an even integer, then $\{0\}$ and $H_n$ are the only closed $\GL_n(\C)$-invariant subspaces of $H_n$, and that otherwise all other invariant closed subspaces are contained in the space of degree $p$ homogeneous polynomials. It remains to show that for every nonzero degree $p$ homogeneous polynomial, every other such polynomial belongs to the linear space spanned by its $\GL_n(\C)$ orbit. This is not difficult.
\end{proof}
\begin{lem}\label{lem:p_even_integer} Let $f_0 \in H_n$ be a nonzero function such that, for every $A \in \GL_n(\C)$, $z \in \C^{n-1} \mapsto f_0 \circ A(1,z)$ is a polynomial in $z,\overline z$. Then $p$ is an even integer and $f_0$ is a homogeneous polynomial of degree $p$.
\end{lem}
\begin{proof}
Let $P \in \C[X_1,\dots,X_{2n-2}]$ such that $f_0(1,z) = P(z,\overline z)$. Using that $f_0 \in H_n$, we have that $|P(z,\overline z)| = O((1+|z|)^p)$, and in particular $\mathrm{deg}(P) \leq p$, so we can write
\[ P(z,\overline z) =\sum_{\alpha,\beta \in \N^d, |\alpha|+|\beta|\leq p} a_{\alpha,\beta} z^\alpha \overline z^{\beta}.\]

Let $c \in \C^{n-1}$ and $A=\begin{pmatrix} 1 & c^*\\ 0 & 1\end{pmatrix}$. Similarly there is $P_c \in \C[X_1,\dots,X_{2n-2}]$ of degree $\leq p$ such that $f\circ A(1,z) = P_c(z)$. Then
\[ P_c(z) = |1+\dual{z}{c}|^p f(1,\frac{z}{1+\dual{z}{c}})  = |1+\dual{z}{c}|^p P(\frac{z}{1+\dual{z}{c}},\frac{\overline{z}}{1+\overline{\dual{z}{c}}}) .\]
We can rewrite this quantity as
\[ \sum_{\alpha,\beta} a_{\alpha,\beta} (1+\dual{z}{c})^{\frac p 2 - |\alpha|} (1+\overline{\dual{z}{c}})^{\frac p 2 - |\beta|} z^\alpha \overline z^{\beta}.\]
By expanding $(1+t)^l=\sum_{n \geq 0} \binom{l}{n} t^n$, for small $z$ the preceding sum is 
\[ \sum_{\alpha,\beta,n,m} a_{\alpha,\beta}\binom{\frac p 2 -|\alpha|}{n}\binom{\frac p 2 -|\beta|}{m} \dual{z}{c}^n \overline{\dual{z}{c}}^m z^\alpha \overline z^\beta.\]
Since $P_c$ is a polynomial of degree $\leq p$, we get that for every $N>p$, 
\[ \sum_{|\alpha|+|\beta|+n+m=N} a_{\alpha,\beta}\binom{\frac p 2 -|\alpha|}{n}\binom{\frac p 2 -|\beta|}{m} \dual{z}{c}^n \overline{\dual{z}{c}}^m z^\alpha \overline z^\beta =0.\]
Since this is valid for every $c$, we get 
\[ a_{\alpha,\beta}\binom{\frac p 2 -|\alpha|}{n}\binom{\frac p 2 -|\beta|}{m} = 0\]
for every $\alpha,\beta \in \N^d$ and $n,n \in \N$ such that $|\alpha|+|\beta|+n+m>p$.

Let $\alpha,\beta$ such that $a_{\alpha,\beta} \neq 0$ (such $\alpha,\beta$ exist by the assumption that $f_0$ is nonzero). Then taking $n=0$ and $m$ very large, we find that $\binom{\frac p 2 -|\beta|}{m} =0$, which implies that $\frac p 2 - |\beta|$ is a nonnegative integer. Similarly $\frac p 2 - |\alpha|$  is a nonnegative integer. This proves that $p$ is an even integer and 
\[ f_0(1,z) = \sum_{|\alpha|,|\beta| \leq \frac p 2} a_{\alpha,\beta} z^\alpha \overline z^{\beta}.\]
By homogeneity we get
\[ f_0(z_1,z) = \sum_{|\alpha|,|\beta| \leq \frac p 2} z_1^{\frac p 2 - |\alpha|} z^\alpha \overline z_1^{\frac p 2 - |\beta|} \overline z^{\beta}.\]
This is the lemma.
\end{proof}
\begin{lem}\label{lem:lebesgue_measure}
The Lebesgue measure $\lambda$ on $\CP^{n-1}$ is given by 
\[ \int_{\CP^{n-1}} F(\C z) d\lambda(z) = c \int_{\C^{n-1}} F(\C(1,z)) \frac{1}{(1+|z_1|^{2}+ \dots + |z_{n-1}|^{2})^n} dz\]
for some number $c>0$.
\end{lem}
\begin{proof} It is a change of variable to compute that the finite measure
  \[ F\in C(\CP^{n-1}) \mapsto \int_{\C^{n-1}} F(\C(1,z)) \frac{1}{(1+|z_1|^{2}+ \dots + |z_{n-1}|^{2})^n} dz\]
  is invariant by $\U(n)$.
\end{proof}
\begin{rem}\label{rem:Wiener_real}
We did not use the full strength of Proposition \ref{prop:wiener} for $\K=\C$, as we used it for a function $g$ satisfying $g(z) = O ((1+|z|)^{-p-d-2})$, which is strictly stronger that the required $g(z) = O ((1+|z|)^{-p-d-1})$. The reason for this $2$ is that the real dimension drops by $2$ between $\C^n$ and $\CP^{n-1}$. In the real case, the dimension drops by $1$, and the same proof (using all the assumptions of Proposition \ref{prop:wiener} this time) leads to the following.

The $\GL_n(\R)$-invariant closed subspace of the Banach space $H_{n,\R}$ of continuous $p$-homogeneous functions $\R^n \to \R$ are (1) $\{0\}$ and $H_{n,\R}$ if $p$ is not an even integer (2) $\{0\}$, $H_{n,\R}$ and the space of homogeneous degree $p$ polynomials if $p$ is an even integer. 

As a consequence, the conclusion of remark \ref{rem:spatial_and_regular_isometries} holds also over $\R$.
\end{rem}

\subsection*{Acknowledgements} The author thanks Jean-Christophe Mourrat for interesting discussions that lead to the proof of Proposition \ref{prop:wiener}, and Alexandros Eskenazis, Mikhail Ostrovskii and Ignacio Vergara for useful comments, suggestions and corrections. Special thanks are due to Gilles Pisier for all that, and also for encouraging the author to think about the content of Section~\ref{section:comparison}.

\bibliographystyle{plain_correctalpha} 
\bibliography{biblio} 
\end{document}